\documentclass[11pt]{article}
\usepackage{amssymb,amsthm,hyperref,amsmath,bm,amsfonts}
\usepackage{arydshln}
\usepackage{verbatim}
\usepackage{graphicx}
\usepackage{float}
\usepackage{subfigure}
\usepackage{multirow}
\usepackage{subfigure}
\usepackage{color}
\usepackage{marvosym}

\theoremstyle{plain}

\newtheorem{theorem}{Theorem}[section]
\newtheorem{lemma}[theorem]{Lemma}
\newtheorem{prop}{Proposition}[section]
\newtheorem{definition}{Definition}[section]
\newtheorem{assumption}{Assumption}[section]
\newtheorem{remark}{Remark}[section]
\numberwithin{equation}{section}

\title{Construction of Boundary Conditions for Hyperbolic Relaxation Approximations\\ 
                   II: Jin-Xin Relaxation Model}

\author{Xiaxia Cao$^{*}$\\
$Department$ $of$ $Mathematical$ $Sciences,$\\
$Tsinghua$ $University,$ $Beijing$ $100084,$ $China.$\\
\\
Wen-An Yong$^{\dagger}$\\
$Department$ $of$ $Mathematical$ $Sciences,$\\
$Tsinghua$ $University,$ $Beijing$ $100084,$ $China.$
}

\date{\today}

\begin{document}

\maketitle{}

\begin{abstract}
This is our second work in the series about constructing boundary conditions for hyperbolic relaxation approximations. The present work is concerned with the one-dimensional linearized Jin-Xin relaxation model, a convenient approximation of hyperbolic conservation laws, 
with non-characteristic boundaries. Assume that proper boundary conditions are given for the conservation laws. We construct boundary conditions for the relaxation model with the expectation that the resultant initial-boundary-value problems are approximations to the given conservation laws with the boundary conditions. The constructed boundary conditions are highly non-unique. Their satisfaction of the generalized Kreiss condition is analyzed. The compatibility with initial data is studied. Furthermore, by
resorting to a formal asymptotic expansion, we prove the effectiveness of the approximations. 
\end{abstract}

\footnote{
$*$E-mail: caoxx18@mails.tsinghua.edu.cn

\quad$\dagger$E-mail: wayong@tsinghua.edu.cn
}


\section{Introduction}

This is our second work in the series about constructing boundary conditions~(BCs) for hyperbolic relaxation systems, which are an important class of partial differential equations. They describe a large number of various non-equilibrium phenomena. Important examples arise in chemically reactive flows~\cite{ern1994multicomponent}, the kinetic theory~\cite{cai2015a,gatignol1975theorie,levermore1996moment,mieussens2000discrete}, compressible viscoelastic flows~\cite{chakraborty2015constitutive,yong2014newtonian}, traffic flows~\cite{10.2307/118533,whitham2011linear}, thermal non-equilibrium flows \cite{1986vincentintroduction} and so on.\vspace{1ex}

On the other hand, relaxation systems also arise as convenient approximations of hyperbolic conservation laws~\cite{doi:10.1137/S0036142998343075,bouchut1999construction,jin1995the}, say
\begin{equation}\label{1.1}
\partial_t\bm u+\partial_xf(\bm u)=0,
\end{equation}
where $\bm u=\bm u(x,t)\in \mathbb{R}^n.$
 A typical example is the Jin-Xin relaxation model~\cite{jin1995the} 
\begin{align}\label{1.2}
        \partial_t\bm u+ \partial_x\bm v&=0,\nonumber\\[2mm]
       \partial_t\bm v+a\partial_x\bm u&=-\frac{1}{\epsilon}\left(\bm v-f(\bm u)\right),
  \end{align}
where $a$ is a positive constant and $\epsilon$ is a small positive parameter called the relaxation rate. This model can provide novel numerical schemes to simulate shock waves without using Riemann solvers.

When the conservation laws \eqref{1.1} are given in a spatial domain with boundaries, say $x>0$, proper BCs are needed at the boundaries for the relaxation model \eqref{1.2} to play its role. To be precise, we notice that the coefficient matrix of the relaxation model has $n$ positive eigenvalues. According to the classical theory for hyperbolic equations~\cite{benzoni2007multi}, $n$ BCs are needed at the boundary $x=0$. On the other hand, the number of the given BCs for the hyperbolic conservation laws is equal to the number of positive eigenvalues of  $\frac{\partial f(\bm u)}{\partial \bm u}$, which is less than $n$ in general. Thus, new BCs are required. This obvious question has not been resolved for a long time. The present paper attempts to answer this question for the one-dimensional relaxation model \eqref{1.2}.

Like the first paper in this series~\cite{zhou2020construction}, this work assumes that the BCs for the conservation laws are given and satisfies the Kreiss condition \cite{higdon1986initial}. Such an assumption is reasonable for the conservation laws are classical and many mathematically correct and physically-based BCs thereof are available. In addition, for simplicity we only consider the case where the spatial domain is the half-space $x>0$. According to \cite{majda1975initial}, such a domain is representative. Furthermore, we assume that the conservation laws are linear and the boundary $x=0$ is non-characteristic  for the conservation laws. The corresponding nonlinear and/or multi-dimensional problems with or without characteristic boundaries are more challenging. They are our on-going project.

Under the above circumstance, the goal of this paper is to construct proper BCs for the relaxation model so that, as the relaxation rate is small, the resultant initial-boundary-value problems are  good approximations to the conservation laws with the given BCs. For this purpose, we firstly resort to asymptotic expansions to obtain certain algebraic relations. The construction is based on these relations and the BC theory developed in \cite{yong1999boundary}. It is highly non-unique. In order to show the effectiveness of the constructed BCs, we partly show that they fulfill the generalized Kreiss condition~(GKC)~\cite{yong1999boundary}, which is essentially necessary for the convergence when $\epsilon$ goes to zero. Furthermore, we prove the convergence for initial data compatible  with the constructed BCs.

At this point, we mention that a difficulty in this work is to verify the GKC. The first article \cite{zhou2020construction} in this series does not directly verify the GKC but proves the strict dissipativeness of the constructed BCs. Other related works in literature all assume that the BCs are prescribed for the hyperbolic relaxation systems~\cite{borsche2018a,cai2018numerical,doi:10.1081/PDE-100106135,nishibata1996the,wang1998asymptotic,xin2000stiff,xu2004boundary}. By contrast, the BCs in the present work are not given.

This paper is organized as follows. Section $2$ contains some preliminaries and a formal asymptotic expansion. The detailed construction of BCs or the main result of this paper is presented in Section $3$. In Section $4$ the GKC is reviewed, while its verification is given in Section $5$. Section $6$ is devoted to the compatibility of the constructed BCs with initial data. In Section $7,$ the formal asymptotic solution is constructed. The effectiveness is showed with error estimates by the energy method and Laplace transformation in Section $8.$  A special case is discussed in the appendix.


\section{Preliminaries}

 We start with the exact equations to be studied in this paper. The one-dimensional linear hyperbolic system of conservation laws reads as
\begin{equation}\label{2.1}
\partial_t\bm u+F\partial_x\bm u=0, \quad x>0,t>0. 
\end{equation}
The hyperbolicity means that the coefficient matrix $F$ can be real diagonalized, that is, there is an invertible matrix $T$ such that 
  $$
  T^{-1}FT=\Lambda=\text{diag}(\lambda_1,\cdots,\lambda_n).
$$ 
We assume that $F$ is invertible, which corresponds to the assumption that the boundary $x=0$ is non-characteristic. Let $\lambda_1,\cdots,\lambda_l > 0, \lambda_{l+1},\cdots,\lambda_n< 0.$ According to the classical theory~\cite{benzoni2007multi}, $l$ BCs of the form
\begin{equation}\label{2.2}
\hat B\bm u(0,t)=\hat b(t)
\end{equation}
are prescribed at $x=0$.
Here $\hat B$ is  an $l\times n$-matrix such that $\hat BR_1^{U}$ is invertible, where $R_1^{U}$ consists of the first $l$ columns of $T$. \vspace{1ex}

The corresponding Jin-Xin relaxation model is 
  \begin{align*}
        \partial_t\bm u+ \partial_x\bm v&=0,\nonumber\\[2mm]
       \partial_t\bm v+\bar A\partial_x\bm u&=-\frac{1}{\epsilon}\left(\bm v-F\bm u\right).
  \end{align*}
This is more general than that in \eqref{1.2} for the matrix $\bar A$ has the form
$$
\bar A=T\text{diag}(a_1,\cdots,a_n)T^{-1}
 $$  
  with $a_j>0$, which includes the case  $ \bar A=aI_n$ and implies $F\bar A=\bar AF.$ 
 Let $\bm p=\bm v-F\bm u.$ The relaxation model above can be rewritten as
 \begin{equation}\label{2.3}
 \left(\begin{array}{cc}
     \bm u \\
     \bm p \\
\end{array}
\right)_t
+
\left(\begin{array}{cc}
       F   & I_n \\
    \bar A-F^2 & -F \\
\end{array}
\right)
\left(\begin{array}{cc}
\bm u \\
\bm p\\
\end{array}
\right)_x
=\frac{1}{\epsilon}\left(\begin{array}{cc}
0 & 0 \\
0 & -I_n\\
\end{array}\right)\left(\begin{array}{cc}
\bm u \\
\bm p\\
\end{array}\right).
\end{equation}
Here and below, $I_k$ is the identity matrix of order $k$.\vspace{1ex}

For this small parameter problem, we seek the following formal asymptotic solutions
\begin{equation}\label{2.4}
  \left(
  \begin{array}{cc}
     \bm u_\epsilon \\
     \bm p_\epsilon\
  \end{array}
  \right)(x,t)=
  \left(\begin{array}{cc}
      \bar{\bm u} \\
       \bar{\bm p}\\
        \end{array}
   \right)(x,t;\epsilon)+\left(\begin{array}{cc}
                                \bm \mu\\
                                 \bm\nu\\
                         \end{array}
                         \right)(x/\epsilon,t;\epsilon).
\end{equation}
Here the first term is the outer solution
\begin{equation}\label{2.5}
  \left(\begin{array}{cc}
       \bar{\bm u} \\
       \bar{\bm p}\\
  \end{array}
  \right)(x,t;\epsilon)
= \left(\begin{array}{cc}
        \bar {\bm u}_0 \\
         \bar {\bm p}_0\\
        \end{array}
        \right)(x,t)+\epsilon\left(\begin{array}{cc}
                                   \bar{\bm  u}_1 \\
                                   \bar{\bm  p}_1\\
                                  \end{array}
                              \right)(x,t),
\end{equation}
while the second term is the boundary-layer correction
\begin{equation}\label{2.6}
\left(\begin{array}{cc}
\bm \mu \\
\bm\nu\\
\end{array}\right)(\xi,t;\epsilon)=\left(\begin{array}{cc}
\bm\mu_0 \\
\bm\nu_0\\
\end{array}\right)(\xi,t)+\epsilon\left(\begin{array}{cc}
\bm\mu_1 \\
\bm\nu_1\\
\end{array}\right)(\xi,t)
\end{equation}
with $\xi=x/\epsilon$. As the boundary-layer corrections, they satisfy the matching conditions
\begin{equation}\label{2.7}		
\bm\mu_j(\infty,t)=\bm\nu_j(\infty,t)=0,\qquad j=0,1.
\end{equation}

The outer solution asymptotically satisfies the relaxation system \eqref{2.3}. We substitute the expansion \eqref{2.5} into the equations in \eqref{2.3} and equate the coefficients of $\epsilon^k$ to obtain
\begin{equation}\label{2.8}
\bar{\bm p}_0=0,
\end{equation}
\begin{equation}\label{2.9}
\bar {\bm u}_{0t}+F\bar{\bm u}_{0x}=0,
\end{equation}
\begin{equation}\label{2.10}
\bar{\bm p}_1=-(\bar A-F^2)\bar {\bm u}_{0x},
\end{equation}
\begin{equation}\label{2.11}
\bar{\bm  u}_{1t}+F\bar{\bm u}_{1x}=-\bar {\bm p}_{1x}.
\end{equation}
 Similarly, substituting the corrections  \eqref{2.6} into the equations in \eqref{2.3} and equating the coefficients of $\epsilon^k$ for $k=\{-1,0\},$ we get
\begin{equation}\label{2.12}
F\bm\mu_{0\xi}+\bm\nu_{0\xi}=0,
\end{equation}
\begin{equation}\label{2.13}
(\bar A-F^2)\bm\mu_{0\xi}-F\bm\nu_{0\xi}=-\bm\nu_0,
\end{equation}
\begin{equation}\label{2.14}
\bm\mu_{0t}+F\bm\mu_{1\xi}+\bm\nu_{1\xi}=0,
\end{equation}
\begin{equation}\label{2.15}
\bm\nu_{0t}+(\bar A-F^2)\bm\mu_{1\xi}-F\bm\nu_{1\xi}=-\bm\nu_1.
\end{equation}

Since $\bm\mu_0(\infty,t)=\bm\nu_0(\infty,t)=0$ in \eqref{2.7}, we integrate the equation in \eqref{2.12} from $\xi$ to $\infty$ to obtain
\begin{equation}\label{2.16}
\bm\mu_0=-F^{-1}\bm\nu_0.
\end{equation}
Substituting this into \eqref{2.13} and integrating from $\xi$ to $\infty,$ we get the equation for $\bm\nu_{0}=\bm\nu_{0}(\xi,t)$:
\begin{equation}\label{2.17}
\bm\nu_{0\xi}=F\bar A^{-1}\bm\nu_0.
\end{equation}
Similarly, we integrate the equation in \eqref{2.14} from $\xi$ to $\infty$  to get
\begin{equation}\label{2.18}
\bm\mu_1=-F^{-1}\bm\nu_1+F^{-1}\int_\xi^\infty \bm\mu_{0t}(s,t)ds.
\end{equation}
Finally, with \eqref{2.14} and \eqref{2.15} we deduce that $\bm\nu_{1}=\bm\nu_{1}(\xi,t)$ satisfies
\begin{equation}\label{2.19}
\bm\nu_{1\xi}=F\bar A^{-1}\bm\nu_1+F^{-1}\bm \nu_{0t}.
\end{equation}

Consequently, we derive the equations for the expansion coefficients $\bar {\bm u}_{0},$ $\bar {\bm u}_{1},$ $\bm\nu_{0}$ and $\bm\nu_{1}$. To determine them and thereby the expansion, proper boundary and initial conditions are needed. This will be discussed in Section $7.$


\section{Construction of Boundary Conditions}

In this section, we construct BCs of the form
\begin{equation}\label{3.1}
B\left(\begin{array}{cc}
      \bm u \\
      \bm p \\
\end{array}
\right)(0,t)=b_{\epsilon}(t)
\end{equation}
for the relaxation system \eqref{2.3},
where  $B=(B_u,B_p)$ is a constant matrix and
$$
b_{\epsilon}(t)=b_0(t)+\epsilon b_1(t)+\epsilon^2 b_2(t).\\
$$
Notice that coefficient matrix  
\begin{equation*}
 A:= \left(\begin{array}{cc}
      F     &   I_n \\
     \bar A-F^2  &  -F \\
   \end{array}
   \right)
\end{equation*}
for the relaxation system has $n$ positive eigenvalues $ \sqrt{a_j}$ $(j=1,2,\cdots,n).$ According to the classical theory~\cite{benzoni2007multi}, $n$ BCs should be given for the relaxation system. Therefore, the boundary matrix $B$ should be a full-rank $n\times 2n$-matrix. In what follows, by a boundary matrix we always mean that it is full-rank. 

Our construction bases on the expectation that the formal asymptotic solution \eqref{2.4}-\eqref{2.6} satisfies the BCs in \eqref{3.1} with $\epsilon=0$:
\begin{equation*}
B\left(\begin{array}{cc}
\bar {\bm u}_0(0,t)+{\bm\mu}_0(0,t) \\
\bar {\bm p}_0(0,t)+{\bm\nu}_0(0,t) \\
\end{array}
\right)=b_{0}(t).
\end{equation*}
From \eqref{2.8} and \eqref{2.16} it follows that
\begin{equation}\label{3.2}
(B_u,B_p)\left(\begin{array}{cc}
\bar {\bm u}_0(0,t)-F^{-1}{\bm\nu}_0(0,t)\\
{\bm\nu}_0(0,t)\\
\end{array}
\right)=b_{0}(t).
\end{equation}
In addition, it is expected that $\bm{\bar u}_0(x,t)$ is the solution to the conservation laws \eqref{2.1} with the BC \eqref{2.2} and certain initial data. Therefore, we require
\begin{equation}\label{3.3}
\hat B\bar{\bm u}_0(0,t)=\hat b(t).
\end{equation}
With \eqref{3.2} and \eqref{3.3}, we construct the boundary matrix $B$ and $b_0(t)$, while $b_1(t)$ and $b_2(t)$ will be  constructed in Section $6$ for compatibility of boundary and initial data. 

When $l=n,$ the coefficient matrix $F$ in \eqref{2.1} has $n$ positive eigenvalues and the boundary matrix $\hat B$ in \eqref{2.2} is invertible, say $\hat B=I_n$. The coefficient matrix $F\bar A^{-1}$ in \eqref{2.17} has only positive eigenvalues and therefore  $\bm \nu_0=0$ is the unique bounded solution thereof. Thus, it is natural to choose 
\begin{equation}\label{3.4}
   B_u=\hat B=I_n,\quad B_p \quad  \text{arbitrary} \quad \text{and} \quad b_0(t)=\hat b(t).
\end{equation}

For $l<n$, we recall that  
\begin{equation*}
T^{-1}FT = \Lambda \triangleq \left(\begin{array}{cc}
\Lambda^{+} & \\
&\Lambda^{-} \\
\end{array}\right)
\end{equation*}
with
\begin{equation*}
\Lambda^{+}=\text{diag}(\lambda_1,\cdots, \lambda_l), \qquad\quad \Lambda^{-}=\text{diag}(\lambda_{l+1},\cdots, \lambda_n).
\end{equation*}
Referring to this partition, we set $T=(R_1^U,R_1^S)$
and
 $
T^{-1}\bar{\bm u}_0(0,t)=\left(\begin{array}{c}
\alpha^+(t)\\
\alpha^-(t)\\
\end{array}
\right).
$
Then we have
\begin{equation}\label{3.5}
\bar{\bm u}_0(0,t)=T\left(\begin{array}{c}
\alpha^+(t)\\
\alpha^-(t)\\
\end{array}
\right)=(R_1^U,R_1^S)\left(\begin{array}{c}
\alpha^+(t)\\
\alpha^-(t)\\
\end{array}\right)=R_1^U\alpha^+(t)+R_1^S\alpha^-(t).
\end{equation}
Since  $\hat BR_1^{U}$ is assumed to be invertible, the BCs \eqref{3.3} with the relation \eqref{3.5} can be rewritten as
\begin{equation}\label{3.6} 
\alpha^+(t)=H\alpha^-(t)+J(t), 
\end{equation} 
where 
\begin{equation*}
H=-(\hat BR_1^U)^{-1}\hat BR_1^S,\qquad J(t)=(\hat BR_1^U)^{-1}\hat b(t).
\end{equation*}
Moreover, the relation \eqref{3.2} becomes
\begin{equation}\label{3.7}
(B_uR_1^U,B_p-B_uF^{-1})\left(\begin{array}{cc}
\alpha^+(t)\\
{\bm\nu}_0(0,t)\\
\end{array}\right)=b_{0}(t)-B_uR_1^S\alpha^-(t).
\end{equation}

In addition, recall the equation \eqref{2.17} for $\bm\nu_0$ with $F\bar A^{-1}=T\text{diag}(\frac{\lambda_1}{a_1},\cdots, \frac{\lambda_n}{a_n})T^{-1}$. For its solution $\bm{\nu}_0(\xi,t)$ to be bounded, the initial value $\bm{\nu}_0(0,t)$ must fulfill 
\begin{equation}\label{3.8}
L_1^U\bm{\nu}_0(0,t)=0,
\end{equation}
where $L_1^U$ consists of the first $l$ rows of 
$T^{-1}=\left(\begin{array}{c}L_1^U
\\L_1^S\\\end{array}\right)$. 
Combining this with \eqref{3.7}, we have
\begin{equation*}
\left(\begin{array}{cc}
B_uR_1^U & B_p-B_uF^{-1}\\
    0    & L_1^U\\
\end{array}\right)
\left(\begin{array}{cc}
\alpha^+(t)\\
{\bm\nu}_0(0,t)\\
\end{array}
\right)=\left(\begin{array}{c}
b_{0}(t)-B_uR_1^S\alpha^-(t)\\
0\\
\end{array}
\right).
\end{equation*}
Referring to Lemma $3.4$ in \cite{yong1999boundary}, we know that the matrix
$$\left(\begin{array}{cc}
  B_uR_1^U & B_p-B_uF^{-1}\\
      0    & L_1^U\\
\end{array}\right)$$
is invertible provided that the boundary matrix in \eqref{3.1} fulfills the generalized Kreiss condition~(GKC)  proposed in \cite{yong1999boundary}. Then we can obtain  $\bm\nu_0(0,t)$ by solving the above algebraic equations. Particularly, $\bm\nu_0(0,t)$ can be expressed as
\begin{equation}\label{3.9}
\bm{\nu}_0(0,t)=C\alpha^-(t)+D(t)
\end{equation}
with $C$ an $n\times (n-l)$ parameter matrix and $D(t)$ a function of $t$, satisfying  
\begin{equation}\label{3.10}
 L_1^UC= L_1^UD(t)=0
\end{equation} 
due to \eqref{3.8}.\vspace{1ex}

With \eqref{3.9} and \eqref{3.6}, the relation \eqref{3.7} becomes
\begin{equation*}
(B_uR_1^U,B_p-B_uF^{-1})
\left(\begin{array}{cc}
H\alpha^-(t)+J(t) \\
C\alpha^-(t)+D(t)\\
\end{array}
\right)
=b_0(t)-B_uR_1^S\alpha^-(t).
\end{equation*}
This holds for any $\alpha^-(t)$ determined by initial data, leading to
\begin{equation}\label{3.11}
b_0(t)=B_uR_1^UJ(t)+(B_p-B_uF^{-1})D(t)
\end{equation}
and
\begin{equation}\label{3.12}
(B_u,B_p)\left(\begin{array}{cc}
R_1^UH+R_1^S-F^{-1}C\vspace{1ex} \\
C\\
\end{array}\right)=0.
\end{equation}
Notice that $C\in \text{span}\{R_1^S\}$ due to \eqref{3.10}. Let $C=R_1^S\tilde C$ with $\tilde C$ an arbitrary $(n-l)\times(n-l)$ square matrix. Then \eqref{3.12} can be rewritten as
\begin{equation}\label{3.13}
\bar BZ\equiv(B_u,B_pR_1^S)\left(\begin{array}{c}
-R_1^U(\hat BR_1^U)^{-1}\hat BR_1^S+R_1^S-F^{-1}R_1^S\tilde C \vspace{1ex}\\
\tilde C\\
\end{array}\right)=0,
\end{equation}
where $\bar B=(B_u,B_pR_1^S)$ and $Z$ is the $(2n-l)\times(n-l)$-matrix.\vspace{1ex} 

Note that $Z$ is a full-rank matrix for any $\tilde C$. To see this, let $Zx=0$ with $x\in \mathbb R^{n-l}.$ Then we have $\tilde Cx=0$
and
\begin{equation*}
[I_n-R_1^U(\hat BR_1^U)^{-1}\hat B]R_1^Sx=0 \qquad\qquad \text{or} \qquad\qquad  R_1^Sx=R_1^U(\hat BR_1^U)^{-1}\hat BR_1^Sx.
\end{equation*}
The last equation means that the left-hand side belongs to $\text{span}\{R_1^S\},$ while the right-hand side is in $\text{span}\{R_1^U\}.$ Since $\text{span}\{R_1^S\}\cap \text{span}\{R_1^U\}=\emptyset,$ it must  be
\begin{equation*}
x=0.
\end{equation*}
Therefore, $Z$ is full-rank.\vspace{1ex}

 Thus, the column full-rank matrix $Z$ can be extended as a base of $\mathbb R^{2n-l}=\text{span}(Z,\bar Z)$ with $\bar Z$ a full-rank $(2n-l)\times n$-matrix. Hence, there exists an invertible matrix $\left(
\begin{array}{c}
\check B\\
\bar B\\
\end{array}
\right)$ such that
\begin{equation*}
\left(\begin{array}{c}
\check B\\
\bar B\\
\end{array}
\right)(Z,\bar Z)=I_{2n-l}.
\end{equation*}
Particularly, we have $\bar BZ=0.$ In this way, we obtain $\bar B=(B_u,B_pR_1^S)$ and thereby $B=(B_u,B_p).$\vspace{1ex}

Obviously, the boundary matrix $B=(B_u,B_p)$ thus constructed is not unique. The non-uniqueness has other two sources.
The first one is that the matrix $Z$ in \eqref{3.13} depends on the free matrix $\tilde C$. On the other hand, once $Z$ is given, there are infinitely many $\bar B$ satisfying $\bar BZ=0.$ However, such $\bar B$ is unique up to an invertible $n\times n$-matrix multiplying $\bar B$ from right. To see this, let  $\bar B_1$ and $\bar B_2$ satisfy $\bar B_1Z=\bar B_2Z=0.$ We can show that the matrices $(Z,\bar B_1^*)$ and $(Z,\bar B_2^*)$ are both  invertible, where the  superscript $*$ denotes the conjugate transpose. In fact, if
\begin{equation*} 
(Z,\bar B_1^*)\left(\begin{array}{c}
\alpha\\
\beta\\
\end{array}
\right)=Z\alpha+\bar B_1^*\beta=0,
\end{equation*}
we multiply with $\bar B_1$ from right to obtain
\begin{equation*} 
\bar B_1\bar B_1^*\beta=0.
\end{equation*}
Because $\bar B_1$ is full-rank, $\bar B_1\bar B_1^*$ is invertible and thereby $\beta=0$. Moreover, since the matrix $Z$ is full-rank, we have $\alpha=0.$ Hence, the matrix $(Z,\bar B_1^*)$ is invertible. Consequently, $\bar B_2^{*}$ can be expressed as
\begin{equation*} 
\bar B_2^*=(Z,\bar B_1^*)\left(\begin{array}{c}
\gamma \\
\chi\\
\end{array}\right)=Z\gamma+\bar B_1^*\chi.
\end{equation*}
Multiplying with $Z^{*}$ from right we get $0=Z^{*}Z\gamma+0,$
which implies $\gamma=0.$ This indicates that $\bar B_2=\chi^{*}\bar B_1.$ Furthermore, $\chi^{*}$ is invertible for both $\bar B_1$ and $\bar B_2$ are full-rank.\vspace{1ex}

 In summary, for $l<n$ we have constructed the boundary matrix $B=(B_u,B_p)$ satisfying \eqref{3.13} with the freedoms above. Once $B$ is chosen, the right-hand side $ b_0(t)$ is determined completely with \eqref{3.11}. For $l=n$, the construction is given in \eqref{3.4}. These are the main results of this paper.


\section{Generalized Kreiss Condition}
According to~\cite{yong1999boundary}, the boundary matrix $B$ constructed in Section $3$ should satisfy the generalized Kreiss condition~(GKC), which is essentially necessary to have a well-behaved limit when $\epsilon$ goes to zero. Thus, we review the GKC in this section. To do this, we recall the definition of a right-stable matrix for a square matrix.
\begin{definition}
Let $N\times N$-matrix $E$ have precisely $k$~$(0\leq k\leq N)$ stable eigenvalues. A full-rank $N\times k$-matrix $R_E^S$ is called a right-stable matrix of $E$ if
\begin{equation*}
ER_E^S=R_E^SS_-
\end{equation*}
with $S_-$ a $k\times k$ stable matrix.
\end{definition}

Notice that the coefficient matrix $A$ in \eqref{2.3} is invertible. We define 
\begin{equation*}
M(\eta,\xi_0)=A^{-1}(\eta S-\xi_0I_{2n})
\end{equation*}
for parameters $\eta\geq 0$ and $\xi_0\in \mathbb C$ with $\text{Re}\xi_0>0$.
 Here $S=\text{diag}(0,-I_n).$ Referring to Lemma $2.3$ in \cite{yong1999boundary}, we know that $M(\eta,\xi_0)$ has $n$ stable eigenvalues under the sub-characteristic condition
 \begin{equation}\label{4.1}
a_j\geq\lambda_j^2.
\end{equation}
 Here we give a direct proof of this fact. For this purpose, we recall the coefficient matrix
\begin{equation*}
   A=\left(
   \begin{array}{cc}
       F     &   I_n \\
     \bar A-F^2  &  -F \\
   \end{array}
    \right)
\end{equation*}
and use $F\bar A=\bar AF$ to rewrite 
 \begin{align*}
M=M(\eta,\xi_0) & = A^{-1}(\eta S-\xi_0I_{2n})\\[3mm]
  & = \left(\begin{array}{cc}
    \bar A^{-1}     &    0   \\
          0         &   \bar A^{-1}  \\
      \end{array}\right)\left(\begin{array}{cc}
          F      &    I_n   \\
        \bar A-F^2  &   -F  \\
      \end{array}\right)\left(\begin{array}{cc}
       -\xi_0 I_n  &       0           \\
            0      & -(\eta+\xi_0)I_n   \\
       \end{array}\right)\\[3mm]
  & = \left(
       \begin{array}{cc}
        -\xi_0 \bar A^{-1}F  & -(\eta+\xi_0)\bar A^{-1}\\
       -\xi_0(I_n-\bar A^{-1}F^2) & (\eta+\xi_0)\bar A^{-1}F\\
       \end{array}
        \right).
\end{align*}
Moreover, with $F=T\Lambda T^{-1}$ we have
 \begin{equation}\label{4.2}
\left(\begin{array}{cc}
T^{-1} &     0  \\
   0   &  T^{-1} \\
\end{array}	\right)M 
\left(\begin{array}{cc}
T & 0\\
0 & T\\
\end{array}\right)=\left(
\begin{array}{cc}
-\xi_0 (\bar\Lambda)^{-1}\Lambda       &  -(\eta+\xi_0)(\bar\Lambda)^{-1}    \\
-\xi_0(I_n-\left(\bar\Lambda)^{-1}\Lambda^2\right) &  (\eta+\xi_0)(\bar\Lambda)^{-1}\Lambda\\
\end{array}\right),
\end{equation}
where $\bar\Lambda=\text{diag}(a_1,\cdots,a_n)$ and $\Lambda=\text{diag}(\lambda_1,\cdots,\lambda_n).$ It is known that the eigenvalues of matrix $M$ are equal to the eigenvalues of the following matrix:
 \begin{equation*}
\frac{1}{a_j}\left(\begin{array}{cc}
-\xi_0\lambda_j      & -(\eta+\xi_0)\\
\xi_0(\lambda_j^2-a_j) & \lambda_j(\eta+\xi_0)\\
\end{array}
\right),\qquad j=1,2\cdots,n.
\end{equation*}
For this $2\times 2$-matrix, the corresponding characteristic polynomial is
\begin{equation*}
\lambda^2-\frac{\eta \lambda_j}{a_j}\lambda-\frac{(\eta+\xi_0)\xi_0}{a_j}=0.
\end{equation*}
Denote by $(\kappa_{j})_{\pm}$ the two solutions of the last equation. We have

\begin{lemma}
Under the sub-characteristic condition, it holds that $\text{Re}(\kappa_{j})_{+}\text{Re}(\kappa_{j})_{-}<0$ for each $j.$ 
\end{lemma}
\begin{proof}

Firstly, we show that $\text{Re}(\kappa_{j})_{+}\text{Re}(\kappa_{j})_{-}\neq 0.$  Otherwise, we may assume that $(\kappa_j)_{+}=ib$ with $b$ a real number. It follows that
\begin{equation}\label{4.3}
(\kappa_{j})_++(\kappa_{j})_-=\frac{\eta \lambda_j}{a_j},
\qquad\quad (\kappa_{j})_+(\kappa_{j})_-=-\frac{(\eta+\xi_0)\xi_0}{a_j},
\end{equation} 
and thereby $(\kappa_j)_{-}=\frac{\eta\lambda_j}{a_j}-ib.$ Set $\xi_0=\alpha+i\beta$ with $\alpha>0.$ We deduce that
\begin{equation*}\label{7.18}
(\kappa_j)_{+}(\kappa_j)_{-}=b^2+i\frac{\eta\lambda_j}{a_j}b=\frac{\beta^2-(\eta+\alpha)\alpha}{a_j}-i\frac{(\eta+2\alpha)\beta}{a_j}
\end{equation*}
and, therefore,
\begin{equation*}\label{7.20}
\frac{\eta\lambda_j}{a_j}b=-\frac{(\eta+2\alpha)\beta}{a_j},\qquad
b^2=\frac{\beta^2-(\eta+\alpha)\alpha}{a_j}.
\end{equation*}
These lead to
\begin{equation*}\label{7.22}
(\eta+2\alpha)^2\frac{\beta^2}{\lambda_j^2}=\eta^2\frac{\beta^2-(\eta+\alpha)\alpha}{a_j}.
\end{equation*}
By the sub-characteristic condition \eqref{4.1}, we see that 
\begin{equation*}
 (\eta+2\alpha)^2\beta^2\leq\eta^2(\beta^2-(\eta+\alpha)\alpha)<\eta^2\beta^2,
\end{equation*}
which is impossible. Thus, we have shown that $\text{Re}(\kappa_{j})_{+}\text{Re}(\kappa_{j})_{-}\neq 0.$\vspace{1ex}

On the other hand, it is clear that $\text{Re}(\kappa_{j})_{+}\text{Re}(\kappa_{j})_{-}<0$ when $\xi_0=1$ and $\eta=0.$ Hence, by the continuity of $(\kappa_{j})_{\pm}$ with respect to the coefficients, we have $\text{Re}(\kappa_{j})_{+}\text{Re}(\kappa_{j})_{-}<0$ for all $\eta\geq 0$ and complex number $\xi_0$ with $\text{Re}\xi_0>0.$ This completes the proof.
\end{proof}

According to the above fact, the right-stable matrix $R_M^S(\eta,\xi_0)$ of $M(\eta,\xi_0)$ is a $2n\times n$-matrix. Note that the boundary matrix $B$ is an $n\times2n$ full-rank matrix.   Then the GKC can be stated as~\cite{yong1999boundary}: there  exists a constant $c_K>0$ such that
\begin{equation}\label{4.4}
\left|\text{det}\{BR_M^S(\eta,\xi_0)\}\right|\geq c_K\sqrt{\text{det}\{R_M^{S*}(\eta,\xi_0)R_M^S(\eta,\xi_0)\}}
\end{equation}
for all $\eta\geq 0$ and $\xi_0$ with $\text{Re} \xi_0>0.$ Here the  superscript $*$ denotes the conjugate transpose. \vspace{1ex} 

In order to verify the GKC for the boundary matrix $B$ constructed in Section $3,$ we need a detailed expression of $R_M^S(\eta,\xi_0).$ As above, let $(\kappa_{j})_+$ denote unstable eigenvalues of $M$ with $\text{Re}(\kappa_j)_{+}>0$ and $(\kappa_{j})_-$ stand for stable eigenvalues with $\text{Re}(\kappa_j)_{-}<0~(j=1,2\cdots,n).$ Then, according to \eqref{4.2} and \eqref{4.3} we have 
\begin{align*}
&\left(\begin{array}{cc}
T^{-1} &     0  \\
   0   &  T^{-1} \\
\end{array} \right) M \left(\begin{array}{cc}
                             T & 0\\
                             0 & T\\
                            \end{array}\right)
                      \left(\begin{array}{cc}
                                    e_j\\
                       \Big(\frac{a_j}{\eta+\xi_0}(\kappa_j)_+-\lambda_j\Big)e_j\\
                      \end{array}
                      \right)\\[3mm]
=& \left(
   \begin{array}{cc}
      -\xi_0 (\bar\Lambda)^{-1}\Lambda            &  -(\eta+\xi_0)(\bar\Lambda)^{-1}    \\
    -\xi_0(I_n-\left(\bar\Lambda)^{-1}\Lambda^2\right) &  (\eta+\xi_0)(\bar\Lambda)^{-1}\Lambda\\
   \end{array}\right) \left(\begin{array}{cc}
                        e_j\\
                   \left(\frac{a_j}{\eta+\xi_0}(\kappa_j)_+-\lambda_j\right)e_j\\
                   \end{array}\right)\\[3mm]
=& \left(\begin{array}{cc}
        \left(-(\kappa_j)_++\frac{\lambda_j\eta}{a_j}\right)e_j\\
    \left(-\xi_0-\frac{\eta\lambda_j^2}{a_j}+\lambda_j(\kappa_j)_+\right)e_j\\
                   \end{array}\right)  \\[4mm]
=& \left(\begin{array}{cc}
        (\kappa_j)_-e_j\\
     \left(-\xi_0-\lambda_j(\kappa_j)_-\right)e_j\\
                   \end{array}\right) \\[3mm]
=& \left(\begin{array}{cc}
        (\kappa_j)_-e_j\\
     \left(\frac{a_j}{\eta+\xi_0}(\kappa_j)_-(\kappa_j)_+-\lambda_j(\kappa_j)_-\right)e_j\\
                   \end{array}\right) \\[3mm]
=& (\kappa_j)_-\left(
\begin{array}{cc}
e_j\\
(\frac{a_j}{\eta+\xi_0}(\kappa_j)_+-\lambda_j)e_j\\
\end{array}
\right).
\end{align*}
Thus, $R_{M}^{S}(\eta,\xi_0)$ can be taken to be
\begin{align}\label{RMS}
R_{M}^{S}(\eta,\xi_0)&=\left(\begin{array}{cc}
T & 0\\
0 & T\\
\end{array}\right)
\left(\begin{array}{ccc}
e_1 & \quad \dots  & \quad  e_n\\
\left(\frac{a_1}{\eta+\xi_0}(\kappa_1)_+-\lambda_1\right)e_1 & \quad \cdots &  \quad \left(\frac{a_n}{\eta+\xi_0}(\kappa_n)_{+}-\lambda_n\right)e_n
\end{array}\right)\nonumber\\[2mm]
&\equiv \left(\begin{array}{cc}
T & 0\\
0 & T\\
\end{array}\right)\left(\begin{array}{cc}
I_n\\
Q\\
\end{array}\right),
\end{align}
where $e_j$ represents the j-th column of the identity matrix $I_n$ and $Q=\text{diag}(q_1,q_2,\cdots, q_n)$ with $q_j=\frac{a_j}{\xi_0+\eta}(\kappa_j)_{+}-\lambda_j~(j=1,2\cdots,n).$ Consequently, we obtain
\begin{equation*}
BR_{M}^{S}(\eta,\xi_0)=(B_u,B_p)\left(\begin{array}{cc}
T & 0\\
0 & T\\
\end{array}\right)\left(\begin{array}{cc}
I_n\\
Q\\
\end{array}\right)=\tilde B_u+\tilde B_pQ.
\end{equation*}
Here $\tilde B_u=B_uT$ and $\tilde B_p=B_pT.$\vspace{1ex} 

Set $\tilde B=(\tilde B_u,\tilde B_p)$ and
\begin{equation*}
\tilde R_M^{S}(\eta,\xi_0)=\left(\begin{array}{c}
I_n\\Q
\end{array}\right).
\end{equation*} 
Then the GKC \eqref{4.4} can be rewritten as
\begin{equation}\label{4.5}
\frac{\left|\text{det}\{\tilde B\tilde R_{M}^{S}(\eta,\xi_0)\}\right|}{\sqrt{\text{det}\{\tilde R_M^{S*}(\eta,\xi_0)\tilde R_M^S(\eta,\xi_0)\}}}\geq c_K>0,
\end{equation}
where Lemma 3.3 in \cite{yong1999boundary} is used.
And the denominator can be calculated as  
\begin{equation*}
\text{det}\{\tilde R_M^{S*}(\eta,\xi_0)\tilde R_M^S(\eta,\xi_0)\}=\text{det}\{I_n+Q^*Q\}=\prod_{j=1}^n(1+|q_j|^2).
\end{equation*}

Here is a uniform estimate for $q_j=\frac{a_j}{\xi_0+\eta}(\kappa_j)_{+}-\lambda_j$ depending on the parameters $\eta$ and $\xi_0.$
\begin{lemma}\label{lemma 4.1}
Under the  sub-characteristic condition, we have the  uniform estimate
\begin{align*}
 \left|q_j(\eta,\xi_0)\right|=\left|\frac{a_j}{\xi_0+\eta}(\kappa_j)_{+}-\lambda_j\right|\leq (\sqrt 2+1)\sqrt a_j,\quad &j=1,2,\cdots,n.
\end{align*}
Under the strict sub-characteristic condition $a_j>\lambda_j^2$, there is a positive constant $c$ such that
\begin{align*}
 |q_j(\eta,\xi_0)|\geq c
\end{align*}
for $j>l$, all $\eta\geq 0$ and all complex number $\xi_0$ with $\text{Re}\xi_0>0.$

\end{lemma}

\begin{proof}
For $\eta\geq 0$ and $\text{Re} \xi_0>0,$ under the sub-characteristic condition we have
     \begin{align*}
|q_j|&=\left|\frac{\sqrt{4a_j\xi_0^2+4a_j\xi_0\eta+\lambda_j^2\eta^2}-\lambda_j(\eta+2\xi_0)}{2(\xi_0+\eta)}\right |\\[4mm]
 &\leq\left|\frac{\sqrt{4a_j\xi_0^2+4a_j\xi_0\eta+\lambda_j^2\eta^2}-\lambda_j(\eta+2\xi_0)}{\eta+2\xi_0}\right|\\[4mm]
     & =\left |\sqrt{a_j+(\lambda_j^2-a_j)(\frac{\eta}{\eta+2\xi_0})^2}-\lambda_j\right|\\[4mm]
     & \leq \sqrt {2a_j-\lambda_j^2}+\sqrt a_j\\[3mm]
     & \leq (\sqrt 2+1)\sqrt a_j.
    \end{align*}

For the lower bound, we  refer to \cite{xin2000stiff} and know that $q_j=h_j-\lambda_j$ is analytic in $\text{Re}\xi_0>0$ for each $\eta\geq 0$ under the sub-characteristic condition, where 
$$
h_j(\xi_0)=\frac{a_j}{\xi_0+\eta}(\kappa_j)_{+}=\frac{\eta\lambda_j+\sqrt{4a_j\xi_0^2+4a_j\xi_0\eta+\lambda_j^2\eta^2}}{2(\xi_0+\eta)}.
$$ 
From this we can solve
\begin{equation*}
\xi_0=\frac{\eta h_j(\xi_0)\left(h_j(\xi_0)-\lambda_j\right)}{a_j-h_j^2(\xi_0)}
\end{equation*}
under the strict sub-characteristic condition $\lambda_j<\sqrt a_j$.
This indicates that $h_j(\xi_0)$ is a univalent analytic function. According to the conformal mapping theorem, $h_j=h_j(\xi_0)$  maps the half plane $\text{Re} \xi_0>0$ to a simply connected closed bounded domain $\Omega \subset \mathbb C$ and maps the imaginary axis $\text{Re}\xi_0=0$ to the boundary of $\Omega.$ Let $\xi_0=\theta_1+i\theta_2$ with $\theta_1>0.$ The boundary is parametrized as 
     \begin{equation*}
        h_j(i\theta_2)=\frac{\eta\lambda_j+\sqrt{-4a_j\theta_2^2+4a_j\eta i\theta_2+\lambda_j^2\eta^2}}{2(i\theta_2+\eta)},
      \end{equation*}
 which is a closed curve and intersects the real axis only at $\theta_2=0$ and $\theta_2=\pm\infty$:
     \begin{equation*}
      h_j(0)=\frac{|\lambda_j|+\lambda_j}{2},\quad h_j(\pm\infty)=\sqrt a_j.
      \end{equation*}
When $j>l,$ we have $\lambda_j<0$, $h_j(0)=0$ and therefore $\lambda_j\notin\Omega$. Because $\Omega$ is closed, there exists a positive constant $c$ such that
      \begin{equation*}
      |q_j(\eta,\xi_0)|=|h_j(\eta,\xi_0)-\lambda_j|\geq c,\quad j=l+1,\cdots, n.
      \end{equation*}
This completes the proof.
\end{proof}


\section{Verification of the Generalized Kreiss Condition}

In this section, we verify the GKC for the constructed boundary matrix $ B=(B_u,B_p)$. According to Lemma~\ref{lemma 4.1}, the determinant $\text{det}\{\tilde R_M^{S*}(\eta,\xi_0)\tilde R_M^S(\eta,\xi_0)\}$ in \eqref{4.5} has a positive upper bound. Thus, it suffices to show that $|\text{det}\{\tilde B\tilde R_{M}^{S}(\eta,\xi_0)\}|$ has a positive lower bound.\vspace{1ex}

To this end, we recall the constraint~\eqref{3.13} for the boundary matrix $B=(B_u,B_p)$:
\begin{equation*}
B_pR_1^S\tilde C+B_u(-F^{-1}R_1^S\tilde C+R_1^S+R_1^UH)=0
\end{equation*}
with $H=-(\hat BR_1^U)^{-1}\hat BR_1^S.$ In terms of $\tilde B_u=B_uT$ and $\tilde B_p=B_pT,$ this constraint can be rewritten as
\begin{equation*}
\tilde B_p(T^{-1}R_1^S\tilde C)+\tilde B_u\big(-\Lambda^{-1}(T^{-1}R_1^S\tilde C)+T^{-1}R_1^S+T^{-1}R_1^UH\big)=0.
\end{equation*}
Because $I_n=T^{-1}T=T^{-1}(R_1^U,R_1^S),$ we have

\begin{equation*}
T^{-1}R_1^U=\left(\begin{array}{cccc}
I_l\\
0
\end{array}\right),\qquad T^{-1}R_1^S=\left(\begin{array}{cccc}
0\\
I_{n-l}
\end{array}\right).
\end{equation*}
Consequently, the constraint becomes
\begin{equation}\label{5.1}
\tilde B_p\left(\begin{array}{cccc}
0\\
\tilde C
\end{array}\right)+\tilde B_u\left(\begin{array}{cccc}
H\\
I_{n-l}-\Lambda_-^{-1}\tilde C
\end{array}\right)=0
\end{equation}
with $\Lambda_{-}=\text{diag}(\lambda_{l+1},\cdots,\lambda_n)<0$ and $\tilde C$ a free $(n-l)\times (n-l)$-matrix.\vspace{1ex}

Next we proceed for different $n$  or $l$.


\subsection{$n=1$}
In this case, the coefficient matrix $F$ in \eqref{2.1} is a number and $T=1$. When $F>0,$ the constructed BC in \eqref{3.4} reads as
\begin{equation}\label{5.2}
B_u=\hat B=1,\qquad  B_p\quad  \text{arbitrary}.
\end{equation}
Then
$$
\text{det}\{\tilde B\tilde R_{M}^{S}(\eta,\xi_0)\}=\tilde B_u+\tilde B_pq=B_uT+B_pTq=1+B_pq,
$$
where $q=\frac{a_1}{\xi_0+\eta}(\kappa_1)_{+}-F.$  When $ B_p=0$, the GKC obviously holds. When $B_p\neq0$, referring to the proof of Lemma~\ref{lemma 4.1}, we know that $|\text{det}\{\tilde B\tilde R_{M}^{S}(\eta,\xi_0)\}|$ has a positive lower bound if 
\begin{equation*}
\frac{1}{B_p}\notin [F-\sqrt a_1,0].
\end{equation*}
This indicates that $B_p>\frac{1}{F-\sqrt a_1}$.\vspace{1ex}

For  $F<0$, the constraint~\eqref{5.1} becomes
\begin{equation*}
\tilde B_p\tilde C+\tilde B_u(1-F^{-1}\tilde C)=0
\end{equation*}
with number $\tilde C$ to be determined. If $\tilde C=F$, this constraint implies that $\tilde B_p=0$. Then we must have $\tilde B_u=1$ in order that the boundary matrix $B=(B_u,B_p)$ is full-rank. Thus the determinant 
$$
\text{det}\{\tilde B\tilde R_{M}^{S}(\eta,\xi_0)\}=\tilde B_u+\tilde B_pq=1
$$
and thereby the GKC holds. \vspace{1ex}

If $\tilde C\neq F$, it follows from the constraint that 
\begin{equation*}
\tilde B_u=\frac{F\tilde C}{\tilde C-F}\tilde B_p.  
\end{equation*}
Thus we must have $\tilde B_p\neq 0$ in order that the boundary matrix $B=(B_u,B_p)$ is full-rank.
Then 
$$
\text{det}\{\tilde B\tilde R_{M}^{S}(\eta,\xi_0)\}=\tilde B_u+\tilde B_pq=\tilde B_p\left(\frac{F\tilde C}{\tilde C-F}+q\right),
$$
where $q=\frac{a_1}{\xi_0+\eta}(\kappa_1)_{+}-F.$ 
Referring to the proof of Lemma~\ref{lemma 4.1}, this determinant has a positive lower bound if
\begin{equation*}
\frac{F\tilde C}{\tilde C-F}\notin [F-\sqrt a_1,F].
\end{equation*}
This means that $ \tilde C\in\big(F-\frac{F^2}{\sqrt a_1},F\big)\cup\big(F,+\infty\big).$ Combining the above discussions, we have the following conclusion.
\begin{prop}\label{prop 5.1}
For $n=1$, when $F>0$, the GKC holds for the construction \eqref{5.2} with $B_p>\frac{1}{F-\sqrt a_1}$.
When $F<0$, the GKC holds for the full-rank matrix $ B=(B_u,B_p)$ satisfying the constraint \eqref{5.1} with
$$
\tilde C>F-\frac{F^2}{\sqrt a_1}.
$$
\end{prop}


\subsection{$n=2,~l=1$}

In this case, we partition $\tilde B_u=(\tilde B_{u1},\tilde B_{u2})$ and $\tilde B_p=(\tilde B_{p1},\tilde B_{p2})$. Then the constraint \eqref{5.1} becomes
\begin{equation}\label{5.3}
(\tilde B_{p1},\tilde B_{p2})\left(\begin{array}{cccc}
0\\
\tilde C
\end{array}\right)+(\tilde B_{u1},\tilde B_{u2})\left(\begin{array}{cccc}
H\\
1-\frac{\tilde C}{\lambda_2}
\end{array}\right)=0.
\end{equation}
Notice that $H$ and $\tilde C$ are numbers.\vspace{1ex} 

If $\tilde C= 0,$ it follows from \eqref{5.3} that 
$
\tilde B_{u2}=-\tilde B_{u1}H.
$
Then we have
\begin{align*}
\tilde B\tilde R_{M}^{S}(\eta,\xi_0)
&=\tilde B_u+\tilde B_p\left(\begin{array}{cc}
                              Q^+ & 0\\
                                0 & Q^-
                             \end{array}
                        \right)\\[4mm]
&=\left(\tilde B_{u1},-\tilde B_{u1}H\right)+(\tilde B_{p1},\tilde B_{p2})\left(\begin{array}{cc}
                              Q^+ & 0\\
                                0 & Q^-
                             \end{array}
                        \right)\\[4mm]
&=\left(\tilde B_{u1}+\tilde B_{p1}Q^+,-\tilde B_{u1}H+\tilde B_{p2}Q^-\right)
\end{align*}
and
\begin{align*}
&\text{det}\{\tilde B\tilde R_{M}^{S}(\eta,\xi_0)\}\\[4mm]
=&\text{det}\{(\tilde B_{u1},\tilde B_{p2})\}Q^-+\text{det}\{(\tilde B_{p1},-\tilde B_{u1}H+\tilde B_{p2}Q^-)\} Q^+.
\end{align*}
Considering that $Q^+$ may vanish, it is necessary for the GKC to be true that $(\tilde B_{u1},\tilde B_{p2})$ is invertible. Thus, it is easy to see the following conclusion.
\begin{prop}
For $n=2,l=1$, the GKC holds for the constructed boundary matrix $ B=(B_u,B_p)$ with $\tilde B=(B_uT,B_pT)
\equiv(\tilde B_{u1},\tilde B_{u2},\tilde B_{p1},\tilde B_{p2})$
 satisfying
 $
\tilde B_{u2}=-H\tilde B_{u1}
$, 
$(\tilde B_{u1},\tilde B_{p2})$ invertible and $\tilde B_{p1}$ close to zero.
\end{prop}\vspace{1ex}

If $\tilde C\neq 0,$ the constraint \eqref{5.3} becomes 
\begin{equation}\label{5.4}
\tilde B_{p2}=-\tilde B_{u1}\frac{H}{\tilde C}+\tilde B_{u2}\frac{\tilde C-\lambda_2}{\lambda_2\tilde C}.
\end{equation}
Then we have
\begin{align*}
\tilde B\tilde R_{M}^{S}(\eta,\xi_0)
&= \tilde B_u+\tilde B_p\left(\begin{array}{cc}
                              Q^+ & 0\\
                                0 & Q^-
                             \end{array}
                        \right)\\[4mm]
&=\left(\tilde B_{u1},\tilde B_{u2}\right)+\left(\tilde B_{p1},-\tilde B_{u1}\frac{H}{\tilde C}+\tilde B_{u2}\frac{\tilde C-\lambda_2}{\lambda_2\tilde C}\right)\left(\begin{array}{cc}
                              Q^+ & 0\\
                                0 & Q^-
                             \end{array}
                        \right)\\[4mm]
&=\left(\tilde B_{u1}+\tilde B_{p1}Q^+,-\tilde B_{u1}\frac{H}{\tilde C}Q^-+\tilde B_{u2}\left(1+\frac{\tilde C-\lambda_2}{\lambda_2\tilde C}Q^-\right)\right)
\end{align*}
and
\begin{align*}
&\text{det}\{\tilde B\tilde R_{M}^{S}(\eta,\xi_0)\}\\[4mm]
&=\text{det}\{(\tilde B_{u1},\tilde B_{u2})\}\left(1+\frac{\tilde C-\lambda_2}{\lambda_2\tilde C}Q^-\right)+\text{det}\{\left(\tilde B_{p1},-\tilde B_{u1}\frac{H}{\tilde C}Q^-+\tilde B_{u2}\left(1+\frac{\tilde C-\lambda_2}{\lambda_2\tilde C}Q^-\right)\right)\} Q^+.
\end{align*}
Considering that $Q^+$ may vanish, it is necessary for the GKC to be true that $(\tilde B_{u1},\tilde B_{u2})=B_uT$ is invertible. Having this, it is not difficult to see the following conclusion.
\begin{prop}
For $n=2,l=1$, the GKC holds for the constructed boundary matrix $ B=(B_u,B_p)$ satisfying that $B_{u}$ is invertible, the first column of $B_pT$ is close to zero, and the second column is given in \eqref{5.4} with
$$
\tilde C\in\left(\lambda_2-\frac{\lambda_2^2}{\sqrt a_2},0\right)\cup(0,+\infty).
$$
\end{prop}


\subsection{$l<n$}

In this general case, the constraint \eqref{5.1} depends on the free $(n-l)\times (n-l)$-matrix $\tilde{C}$ and we will verify the GKC only with $\tilde{C}=0$ or $\tilde{C}=\Lambda_-.$ 

\subsubsection {$\tilde C=0$}
In this case, the constraint \eqref{5.1} becomes
\begin{equation*}
\tilde B_{u}\left(\begin{array}{cccc}
H\\I_{n-l}
\end{array}\right)=0
\end{equation*}
with $H=-(\hat BR_1^U)^{-1}\hat BR_1^S.$  This implies that the rank of $ B_u$ is not larger than $l.$
When $l=0$, we have $\tilde B_u=0$. Then it must be that $\tilde B_p=I_n$ in order that the boundary matrix $\tilde B=(\tilde B_u,\tilde B_p)$ is full-rank. Thus it follows from Lemma~\ref{lemma 4.1} that
$$
|\text{det}\{\tilde B\tilde R_{M}^{S}(\eta,\xi_0)\}|=|\text{det}\{\tilde B_u+\tilde B_pQ\}|=|\text{det}\{Q\}|\geq c^n>0.
$$
Therefore the GKC holds. \vspace{1ex}

For $l>0$, the constraint is $\tilde B_{u2}=-\tilde B_{u1}H$ with $(\tilde B_{u1},\tilde B_{u2})=B_uT=B_u(R_1^U,R_1^S)$. Notice that $\tilde B_{u1}$ is an $n\times l$-matrix. By a linear transformation, we may as well assume that $\tilde B_{u1}$ has the form
\begin{equation*}
\tilde B_{u1}=\left(\begin{array}{c}
\tilde B_{u11}\\
0
\end{array}\right)
\end{equation*}
with $\tilde B_{u11}$ an $l\times l$-matrix. Here the main result is 
\begin{prop}\label{prop 5.4}
The GKC holds for
\begin{equation*}
\tilde B=(B_uT,B_pT)=\left(
\begin{array}{cccc}
\tilde B_{u11} &  -\tilde B_{u11}H   &\quad        \tilde B_{p11}                &\quad           \star        \\
      0        &           0         &\quad          0            &\quad        \tilde B_{p22}
\end{array}
\right)
\end{equation*}
with both $\tilde B_{u11}$ and $(n-l)\times (n-l)$-matrix $\tilde B_{p22}$ invertible, and the spectral radius  
$\rho(\tilde B_{u11}^{-1}\tilde B_{p11})<\frac{1}{\max\limits_{j\leq l}(\sqrt 2+1)\sqrt a_j}.$ Here $\star$ stands for an arbitrary $l\times (n-l)$-matrix.
\end{prop}

\begin{proof}
Set 
$$
Q^{+}=\text{diag}(q_1,\cdots,q_l),\qquad Q^{-}=\text{diag}(q_{l+1},\cdots,q_n).
$$ 
Then we have
\begin{eqnarray*}
\tilde B\tilde R_{M}^{S}(\eta,\xi_0)&=& \tilde B_u+\tilde B_p\left(\begin{array}{cc}
Q^+ & 0\\
0 & Q^-
\end{array}
\right)\\[4mm]
&=&\left(\begin{array}{cc}
\tilde B_{u11} & -\tilde B_{u11}H\\
0 & 0
\end{array}\right)+\left(\begin{array}{cccc}
\tilde B_{p11} & \star\\
 0 & \tilde B_{p22}
\end{array}\right)\left(\begin{array}{cccc}
Q^+ & 0\\
0 & Q^-
\end{array}\right)\\[4mm]
&=&\left(\begin{array}{cc}
\tilde B_{u11}+\tilde B_{p11}Q^+   &    -\tilde B_{u11}H+\star Q^{-}\\
0                                  &        \tilde B_{p22}Q^-
\end{array}\right)
\end{eqnarray*}
and
\begin{equation*}
\text{det}\{\tilde B\tilde R_{M}^{S}(\eta,\xi_0)\}=\left|\tilde B_{u11}\right|\left|I_l+\tilde B_{u11}^{-1}\tilde B_{p11}Q^+\right|\left|\tilde B_{p22}\right|\prod_{j=l+1}^{n}q_j.
\end{equation*}
Since $\left|\tilde B_{u11}\right|\left|\tilde B_{p22}\right|\neq 0$ is independent of the parameters $\eta$ and $\xi_0$, $|q_j(\eta,\xi_0)|~(j>l)$ has a uniform positive lower bound due to Lemma \ref{lemma 4.1}, and $\left|I_l+\tilde B_{u11}^{-1}\tilde B_{p11}Q^+\right|$ has a uniform positive lower bound under the given condition, $|\text{det}\{\tilde B\tilde R_{M}^{S}(\eta,\xi_0)\}|$ has a positive lower bound.  Hence the proof is completed.

\end{proof}

\begin{remark}
For $l=1,$ it is not difficult to see from the above proof that the spectral radius condition can be relaxed as $\tilde B_{u11}^{-1}\tilde B_{p11}>\frac{1}{\lambda_1-\sqrt a_1}.$
\end{remark}

\subsubsection{$\tilde C=\Lambda_-$}
In this case, the constraint \eqref{5.1} becomes
\begin{equation*}
\tilde B_{p}\left(\begin{array}{cccc}
0\\
\Lambda_-
\end{array}\right)+\tilde B_{u}\left(\begin{array}{cccc}
H\\
0
\end{array}\right)=0.
\end{equation*}
When $l=0$, it follows from this constraint that $\tilde B_p=0$. Then we must have $\tilde B_u=I_n$ in order that the boundary matrix $\tilde B=(\tilde B_u,\tilde B_p)$ is full-rank. Thus the determinant 
$$
\text{det}\{\tilde B\tilde R_{M}^{S}(\eta,\xi_0)\}=|\tilde B_u+\tilde B_pQ|=1
$$
and thereby the GKC holds. \vspace{1ex}

For $l>0$,  we have the following result.

\begin{prop}
The GKC holds for
\begin{equation*}
\tilde B=(B_uT,B_pT)=\left(\begin{array}{cccc}
\tilde B_{u11} &            \star    &\quad  \tilde B_{p11}    &\quad     -\tilde B_{u11} H\Lambda_-^{-1}\\
        0      &    \tilde B_{u22}   &\quad         0          &\quad                    0
\end{array}\right)
\end{equation*}
with both $\tilde B_{u11}$ and $\tilde B_{u22}$ invertible, and the spectral radius  
$\rho(\tilde B_{u11}^{-1}\tilde B_{p11})<\frac{1}{\max\limits_{j\leq l}(\sqrt 2+1)\sqrt a_j}.$ Here $\star$ stands for an arbitrary $l\times (n-l) $-matrix.

\end{prop}
\begin{proof}
 Because
\begin{align*}
\tilde B\tilde R_{M}^{S}(\eta,\xi_0)&= \tilde B_u+\tilde B_p\left(\begin{array}{cccc}
Q^+ & 0\\
0 & Q^-
\end{array}\right)\\[3mm]
& =\left(\begin{array}{cccc}
\tilde B_{u11} & \star \\
0 & \tilde B_{u22}
\end{array}
\right)+\left(\begin{array}{cccc}
\tilde B_{p11} & -\tilde B_{u11} H\Lambda_-^{-1}\\
0 & 0
\end{array}\right)\left(\begin{array}{cccc}
Q^+ & 0\\
0   & Q^-
\end{array}\right)\\[3mm]
&=\left(\begin{array}{cccc}
\tilde B_{u11}+\tilde B_{p11}Q^+ & \star-\tilde B_{u11}H\Lambda_-^{-1}Q^-\\
0 & \tilde B_{u22}
\end{array}\right),
\end{align*}
we have
\begin{equation*}
\text{det}\{\tilde B\tilde R_{M}^{S}(\eta,\xi_0)\}=\left|\tilde B_{u11}\right|\left|I_l+\tilde B_{u11}^{-1}\tilde B_{p11}Q^+\right|\left|\tilde B_{u22}\right|.
\end{equation*}
Then the same argument in the proof of Proposition~\ref{prop 5.4} leads to the conclusion. This completes the proof.
\end{proof}


\section{Compatibility}

The BCs constructed in Section $3$ do not guarantee that the initial-boundary-value problems~(IBVPs) have smooth solutions. To clarify this point, we introduce further constraints on the boundary data $b_{\epsilon}(t)$ in \eqref{3.1} so that the initial and boundary data are compatible, up to a certain order, at $(x,t)=(0,0)$. In this and next sections, we only consider the case where $l<n$, while the simple case $l=n$ is studied in the appendix.  

To do this, we denote by $\bm u_0=\bm u_0(x)$ the initial value for the conservation laws \eqref{2.1}. 
Assume that this initial value is compatible, up to order $2$, with the boundary data in \eqref{2.2} for the conservation laws, that is, 
\begin{equation}\label{6.1}
\hat B(-F\partial_x)^i\bm {u}_0(0)=\partial_t^i\hat b(0),\quad i=0,1,2.
\end{equation}
Similarly, the compatibility up to order $2$ for the relaxation system \eqref{2.3} reads as
\begin{equation}\label{6.2}
(B_u,B_p)\left(\begin{array}{cc}
\partial_t^i\bm{u}\vspace{1ex}\\
\partial_t^i\bm{p}\\
\end{array}\right)(0,0)=\partial_t^ib_{\epsilon}(0),\quad i=0,1,2,
\end{equation}
where 
$$
b_{\epsilon}(t)=b_0(t)+\epsilon b_1(t)+\epsilon^2 b_2(t).
$$ 
In view of \cite{zhou2020construction}, we refer to \eqref{2.8} and \eqref{2.10} and choose 
\begin{equation}\label{6.3}
\left(\begin{array}{cc}
\bm u \\
\bm p\\
\end{array}
\right)(x,0)=\left(\begin{array}{cc}
\bm u_0 \\
0\\
\end{array}
\right)+\epsilon\left(\begin{array}{cc}
0 \\
-(\bar A-F^2)\bm{u}_{0x}\\
\end{array}\right)+\epsilon^2\left(\begin{array}{cc}
0\\
\bm p_{02}\\
\end{array}\right)
\end{equation}
as initial data for the relaxation system. Here $\bm p_{02}$ is to be determined.\vspace{1ex} 

Next we deduce from the relaxation system~\eqref{2.3} that
\begin{equation*}
\left(\begin{array}{cc}
\partial_t^i\bm{u}\vspace{1ex}\\
\partial_t^i\bm{p}\\
\end{array}\right)\Big|_{t=0}=\left(-\left(\begin{array}{cc}
F & I_n\\
\bar A-F^2 & -F\\
\end{array}\right)\partial_x+\frac{1}{\epsilon}\left(\begin{array}{cc}
0 & 0 \\
0 & -I_n\\
\end{array}\right)\right)^i\left(\begin{array}{cc}
\bm{u}\\\bm{p}\\
\end{array}\right)\Big|_{t=0}.
\end{equation*}
Using \eqref{6.3}, we have
\begin{equation*}
\left(\begin{array}{cc}
\partial_t\bm{u}\vspace{1ex}\\
\partial_t\bm{p}\\
\end{array}\right)\Big|_{t=0}=\left(\begin{array}{cc}
-F\partial_x\bm u_0 \\
0\\
\end{array}\right)+\epsilon m_{11}+\epsilon^2 m_{12}
\end{equation*}
with
\begin{equation*}
m_{11}=\left(\begin{array}{cc}
(\bar A-F^2)\partial_{xx}\bm{u}_0 \vspace{1ex}\\
-F(\bar A-F^2)\partial_{xx}\bm{u}_0-\bm{p}_{02}\\
\end{array}\right),\qquad
 m_{12}=\left(\begin{array}{cc}
-\partial_x\bm {p}_{02}\vspace{1ex} \\
F\partial_x\bm{p}_{02}\\
\end{array}\right)
\end{equation*}
and
\begin{equation}\label{6.4}
\left(\begin{array}{cc}
\partial_t^2\bm{u}\\
\partial_t^2\bm{p}\\
\end{array}\right)\Big|_{t=0}=\left(\begin{array}{cc}
F^2\partial_{xx}\bm u_0\vspace{1ex} \\
\bm{p}_{02}+2F(\bar A-F^2)\partial_{xx}\bm{u}_0\\
\end{array}\right)+\epsilon m_{21}+\epsilon^2 m_{22},
\end{equation}
where
\begin{equation*}
m_{21}=\left(\begin{array}{cc}
\partial_{x}\bm p_{02}\vspace{1ex} \\
-\bar A(\bar A-F^2)\partial_{xxx}\bm{u}_0-2F\partial_{x}\bm{p}_{02}\\
\end{array}\right),\qquad m_{22}=\left(\begin{array}{cc}
0\vspace{1ex}\\
\bar A\partial_{xx}\bm{p}_{02}\\
\end{array}\right).
\end{equation*}
\vspace{1ex}
In view of \eqref{6.4}, we take
\begin{equation}\label{6.5}
\bm{p}_{02}=-2F(\bar A-F^2)\partial_{xx}\bm{u}_0.
\end{equation}
Consequently, the compatibility condition \eqref{6.2} becomes
\begin{equation}\label{6.6}
\partial_t^ib_{0}(0)=(B_u,B_p)\left(\begin{array}{cc}
(-F\partial_x)^i\bm{u}_0(0)\\
0\\
\end{array}\right)=B_u(-F\partial_x)^i\bm{u}_0(0),\quad i=0,1,2
\end{equation}
and
\begin{equation}\label{6.7}
\partial_t^ib_1(0)=(B_u,B_p)m_{i1},\qquad
\partial_t^ib_2(0)=(B_u,B_p)m_{i2},\qquad i=0,1,2,
\end{equation}
where 
$$
m_{01}=\left(\begin{array}{cc}
0\\
-(\bar A-F^2)\bm{u}_{0x}\\
\end{array}\right),\qquad m_{02}=\left(\begin{array}{cc}
0\\
\bm p_{02}\\
\end{array}\right).
$$

Recall \eqref{3.11} where
\begin{equation*}
b_0(t)=B_uR_1^UJ(t)+(B_p-B_uF^{-1})D(t),
\end{equation*}
 while $b_1(t)$ and $b_2(t)$ have not been determined up to now. Our main result of this section is
\begin{theorem}
For $l<n,$ let the boundary matrix $B=(B_u,B_p)$ be given with the constraint \eqref{3.12}  and $b_0(t)$ be given with \eqref{3.11}. Then the compatibility of order $2$  holds if the initial data are chosen according to \eqref{6.3} with \eqref{6.5}, and the boundary data satisfy \eqref{6.7} and
\begin{equation}\label{6.8}
\partial_t^iD(0)=-\big(0,C\big)T^{-1}(-F\partial_x)^i\bm{u}_0(0),\qquad i=0,1,2.
\end{equation}
\end{theorem}
\begin{proof}

With \eqref{6.7}, we only need to check \eqref{6.6}. 
To do this, it follows from \eqref{3.11} and \eqref{6.8} that
\begin{align*}
\partial_t^ib_0(0)
&=\partial_t^i(B_uR_1^UJ(t))\big|_{t=0}-(B_p-B_uF^{-1})\big(0,C\big)T^{-1}(-F\partial_x)^i\bm{u}_0(0).
\end{align*}
Furthermore, from \eqref{3.6} and \eqref{3.12} we deduce that
\begin{align*}
\partial_t^ib_0(0)
 =& B_uR_1^U\big[\partial_t^i\alpha^+(0)-H\partial_t^i\alpha^-(0)\big]-(B_p-B_uF^{-1})\big(0,C\big)T^{-1}(-F\partial_x)^i\bm{u}_0(0)\\[2mm]
 =& B_uR_1^U\partial_t^i\alpha^+(0)+\big[B_uR_1^S+(B_p-B_uF^{-1})C\big]\partial_t^i\alpha^-(0)\\[2mm]
&-(B_p-B_uF^{-1})\big(0,C\big)T^{-1}(-F\partial_x)^i\bm{u}_0(0)\\[2mm]
 =& B_u\big[R_1^U\partial_t^i\alpha^+(0)+R_1^S\partial_t^i\alpha^-(0)\big]+(B_p-B_uF^{-1})\Big[C\partial_t^i\alpha^-(0)-\big(0,C\big)T^{-1}(-F\partial_x)^i\bm{u}_0(0)\Big].
\end{align*}
Recall \eqref{3.5}  that 
$$
T\left(\begin{array}{c}
\alpha^+(t)\\
\alpha^-(t)\\
\end{array}
\right)=(R_1^U,R_1^S)\left(\begin{array}{c}
\alpha^+(t)\\
\alpha^-(t)\\
\end{array}
\right)=\bm u(0,t)
$$
 with $\bm u=\bm u(x,t)$ the solution to the conservation laws~\eqref{2.1}. Since $\partial_t^i \bm u(0,t)\big|_{t=0}=(-F\partial_x)^i\bm{u}_0(0)$ due to the conservation laws, we have
\begin{align*}
\partial_t^ib_0(0)
&= B_u\partial_t^i\bm{u}(0,0)+(B_p-B_uF^{-1})\Big[C\partial_t^i\alpha^-(0)-\big(0,C\big)T^{-1}\partial_t^i \bm u(0,t)\big|_{t=0}\Big]\\[2mm]
& = B_u\partial_t^i\bm{u}(0,0)+(B_p-B_uF^{-1})\Big[C\partial_t^i\alpha^-(0)-\big(0,C\big)\left(\begin{array}{c}
\partial_t^i\alpha^+(0)\\
\partial_t^i\alpha^-(0)\\
\end{array}
\right)\Big]\\[2mm]
& = B_u(-F\partial_x)^i\bm{u}_0(0).
\end{align*}
This completes the proof.
\end{proof}
\begin{remark}\label{remark 6.1}
Recall that $\bm{\nu}_0(0,t)=C\alpha^-(t)+D(t)$. Then the condition \eqref{6.8} is equivalent to
$$
\partial_t^i\bm{\nu}_0(0,0)=0,\quad i=0,1,2.
$$
In addition, $\bm \mu_0|_{t=0}=0$ for $\bm\mu_0=-F^{-1}\bm\nu_0$.
\end{remark}

In summary, we have introduced new constraints \eqref{6.7} and \eqref{6.8} on $b_{\epsilon}(t)$  such that the initial and boundary data for the relaxation system \eqref{2.3} are compatible, up to order 2, at $(x,t)=(0,0)$.


\section{Formal Asymptotic Solutions}

In order to show the effectiveness of the initial and boundary conditions constructed before, we seek a formal approximate solution to the resultant IBVP in this section. To this end, we follow Section $2$ to fix the asymptotic expansion coefficients in \eqref{2.4}-\eqref{2.6}. Thanks to \eqref{2.8}, \eqref{2.10}, \eqref{2.16} and \eqref{2.18}, we only need to determine $\bm{\bar u}_0,$ $\bm{\bar u}_1,$ $\bm{\nu}_0$  and $\bm{\nu}_1$, solving equations \eqref{2.9},~\eqref{2.11},~\eqref{2.17} and \eqref{2.19}, respectively. \vspace{1ex}

According to \eqref{6.3}, it is natural to take 
\begin{equation}\label{7.1}
      \bm{\bar u}_0(x,0)=\bm{u}_0(x),\quad \bm{\bar u}_1(x,0)=0.
\end{equation}
In addition, we need BCs for $\bm{\bar u}_0$ and $\bm{\bar u}_1$ at the boundary $x=0,$ and initial data for $\bm{\nu}_0$ and $\bm{\nu}_1.$ This can be obtained from the expectation that the asymptotic solution \eqref{2.4}-\eqref{2.6} satisfies the BC~\eqref{3.1}, that is,
\begin{equation*}
(B_u,B_p)\left(
\begin{array}{cc}
\bm{\bar u}_0(0,t)+\bm{\mu}_0(0,t)\vspace{1ex}\\
\bm{\bar p}_0(0,t)+\bm{\nu}_0(0,t)\\
\end{array}
\right)=b_0(t)
\end{equation*}
and 
\begin{equation*}
(B_u,B_p)\left(
\begin{array}{cc}
\bm{\bar u}_1(0,t)+\bm{\mu}_1(0,t)\vspace{1ex}\\
\bm{\bar p}_1(0,t)+\bm{\nu}_1(0,t)\\
\end{array}
\right)=b_1(t)\\.
\end{equation*}
With \eqref{2.8}, \eqref{2.10}, \eqref{2.16} and \eqref{2.18}, the last two relations become
\begin{equation}\label{7.2}
B_u\bm{\bar u}_0(0,t)+(B_p-B_uF^{-1})\bm{\nu}_0(0,t)=b_0(t)
\end{equation}
and
\begin{equation}\label{7.3}
B_u\bm{\bar u}_1(0,t)+(B_p-B_uF^{-1})\bm{\nu}_1(0,t)=b_1(t)+B_p(\bar A-F^2)\bm{\bar u}_{0x}(0,t)-B_uF^{-1}\int_0^{\infty}\bm{\mu}_{0t}(s,t)ds.
\end{equation}

For $l<n$, we refer to Theorem 3.2 in~\cite{yong1999boundary}. If the constructed BC satisfies the GKC, then there exists a full-rank $l\times n$-matrix $\hat B_1$ such that $l\times l$-matrix $\hat B_1B_uR_1^U$ is invertible and $\hat B_1(B_p-B_uF^{-1})R_1^S=0$ with $(R_1^U,R_1^S)=T$. On the other hand, from  \eqref{3.8} we know that $\bm\nu_0(0,t)=R_1^S\alpha$ with $\alpha$ an $(n-l)$-vector. Then we multiply \eqref{7.2} with $\hat B_1$ from right to obtain
\begin{equation}\label{7.4}
\hat B_1B_u\bm{\bar u}_0(0,t)=\hat B_1b_0(t).
\end{equation}
This will be shown to be just the given BC \eqref{2.2}. Consequently, $\bm{\bar u}_0$  can be obtained by solving the given IBVP of the conservation laws \eqref{2.1}.\vspace{1ex} 

To see the equivalence of \eqref{7.4} and \eqref{2.2}, we recall the constraint in \eqref{3.13}:
\begin{equation*}
B_uR_1^UH+B_uR_1^S+(B_p-B_uF^{-1})R_1^S\tilde C=0
\end{equation*}
with $H=-(\hat BR_1^U)^{-1}(\hat BR_1^S)$. Multiplying this with $\hat B_1$ from right we obtain
\begin{align*}
\hat B_1B_u\left(I-R_1^U(\hat BR_1^U)^{-1}\hat B\right)R_1^S=-\hat B_1B_uR_1^U(\hat BR_1^U)^{-1}\hat BR_1^S+\hat B_1B_uR_1^S=0.
\end{align*}
Thus we have 
\begin{align*}
\hat B_1B_uT&=\hat B_1B_u(R_1^U,R_1^S)\\[2mm]
&=(\hat B_1B_uR_1^U(\hat BR_1^U)^{-1}\hat BR_1^U,\hat B_1B_uR_1^U(\hat BR_1^U)^{-1}\hat BR_1^S)\\[2mm]
&=\hat B_1B_uR_1^U(\hat BR_1^U)^{-1}\hat BT
\end{align*}
and thereby
\begin{align}\label{7.5}
\hat B_1B_u=\hat B_1B_uR_1^U(\hat BR_1^U)^{-1}\hat B.
\end{align}
In addition, recall the constraint in \eqref{3.11}:
\begin{equation*}
b_0(t)=B_uR_1^UJ(t)+(B_p-B_uF^{-1})D(t)
\end{equation*}
with $J(t)=(\hat BR_1^U)^{-1}\hat b(t)$ and $D(t)\in \text{span}\{R_1^S\}$ due to \eqref{3.10}.
Then we multiply it with $\hat B_1$ from right to obtain 
\begin{equation}\label{7.6}
\hat B_1b_0(t)=\hat B_1B_uR_1^U(\hat BR_1^U)^{-1}\hat b(t).
\end{equation} 
Since $\hat B_1B_uR_1^U$ is invertible, it is easy to see from \eqref{7.5} and \eqref{7.6} that the reduced BC \eqref{7.4} is equivalent to the BC given in \eqref{2.2}. 

On the other hand, there exists an $(n-l)\times n$-matrix $\hat B_2$ such that 
$\left(\begin{array}{cc}
\hat B_1\\
\hat B_2\\
\end{array}\right)$ 
is invertible  for $\hat B_1$ is full-rank. Multiplying \eqref{7.2} with $\hat B_2$ from right, we get
\begin{equation}\label{7.7}
\hat B_2(B_p-B_uF^{-1})\bm\nu_0(0,t)=\hat B_2\Big(b_0(t)-B_u\bm{\bar u}_0(0,t)\Big).
\end{equation}
According to Lemma $3.4$ in~\cite{yong1999boundary}, $\hat B_2(B_p-B_uF^{-1})R_1^S$ is invertible. Thus, the initial value  $\bm\nu_0(0,t)$ is uniquely determined by \eqref{7.7}.\vspace{1ex}

Similarly, for \eqref{2.19} to have a bounded solution $\bm\nu_1=\bm\nu_1(\xi,t)$, the initial value $\bm{\nu}_1(0,t)$ has to be in the form $\bm\nu_1(0,t)=R_1^S\zeta$ with $\zeta$ an $(n-l)$-vector. Thus we multiply \eqref{7.3} with $\hat B_1$ from right to obtain
\begin{equation*}
\hat B_1B_u\bm{\bar u}_1(0,t)=\hat B_1\Big[b_1(t)+B_p(\bar A-F^2)\bm{\bar u}_{0x}(0,t)-B_uF^{-1}\int_0^{\infty}\bm{\mu}_{0t}(s,t)ds\Big].
\end{equation*}
With this and \eqref{7.1}, we can get the unique solution $\bm{\bar u}_1$ to the IBVP of equation \eqref{2.11}. Having $\bm{\bar u}_1$, we multiply \eqref{7.3} with $\hat B_2$ from right to get
\begin{align*}
\hat B_2(B_p-B_uF^{-1})\bm{\nu}_1(0,t)=&\hat B_2\Big[b_1(t)+B_p(\bar A-F^2)\bm{\bar u}_{0x}(0,t)\\[1mm]
&-B_uF^{-1}\int_0^{\infty}\bm{\mu}_{0t}(s,t)ds-B_u\bm{\bar u}_1(0,t)\Big].
\end{align*}
From this,  we get the initial value $\bm\nu_1(0,t)$.\vspace{1ex}
\begin{remark}
From the last relation, \eqref{6.7}, \eqref{7.1} and Remark \ref{remark 6.1}, we deduce that $\bm{\nu}_1(0,0)=0$. In addition, from Remark \ref{remark 6.1} and \eqref{2.18} it follows that $\bm{\mu}_1|_{t=0}=0$.
\end{remark}
 
\begin{remark}
The boundary and initial data of $\bm{\bar u}_1$ are compatible up to order $1$ if the corresponding boundary and initial data for $\bm{\bar u}_0$ are compatible up to order $2.$
\end{remark}

In conclusion, we  have determined all the coefficients in the expansion \eqref{2.4}-\eqref{2.6} and thus constructed a formal asymptotic solution to the  constructed IBVP of the relaxation system~\eqref{2.3}.


\section{Effectiveness}

In this section, we prove the convergence by estimating the difference between the formal approximate solution and the exact solution to the constructed IBVP of the relaxation system \eqref{2.3}:
\begin{eqnarray}\label{8.1}
 \left\{ \begin{array}{l}\left(\begin{array}{cc}
\bm u^{\epsilon} \\
\bm p^{\epsilon}\\
\end{array}\right)_t+\left(\begin{array}{cc}
F & I_n\\
\bar A-F^2 & -F\\
\end{array}\right)\left(\begin{array}{cc}
\bm u^{\epsilon} \\
\bm p^{\epsilon}\\
\end{array}\right)_x=\frac{1}{\epsilon}\left(\begin{array}{cc}
0 & 0 \\
0 & -I_n\\
\end{array}\right)\left(\begin{array}{cc}
\bm u^{\epsilon} \\
\bm p^{\epsilon}\\
\end{array}\right),\vspace{2ex}\\
B\left(\begin{array}{cc}
\bm u^{\epsilon} \\
\bm p^{\epsilon}\\
\end{array}\right)(0,t)=b_{0}(t)+\epsilon b_1(t)+\epsilon^2b_2(t),\vspace{2ex}\\
\left(\begin{array}{cc}
\bm u^{\epsilon} \\
\bm p^{\epsilon}\\
\end{array}\right)(x,0)=\left(\begin{array}{cc}
\bm u_0 \\
0\\
\end{array}\right)(x)+\epsilon\left(\begin{array}{cc}
0 \\
-(\bar A-F^2)\bm{u}_{0x}\\
\end{array}\right)(x)+\epsilon^2\left(\begin{array}{cc}
0\\
\bm{p}_{02}\\
\end{array}\right)(x).\\
 \end{array} \right.
\end{eqnarray}
\vspace{1ex}

Recall that the formal asymptotic solution is 
\begin{equation*}
\left(\begin{array}{cc}
\bm u_\epsilon \\
\bm p_\epsilon\\
\end{array}\right)(x,t)=\left(\begin{array}{cc}
\bar{\bm u}_0 \\
\bar{\bm p}_0\\
\end{array}\right)(x,t)+\epsilon\left(\begin{array}{cc}
\bar{\bm u}_1 \\
\bar{\bm p}_1\\
\end{array}\right)(x,t)+\left(\begin{array}{cc}
\bm \mu_0\\
\bm\nu_0\\
\end{array}\right)(x/\epsilon,t)+\epsilon\left(\begin{array}{cc}
\bm \mu_1\\
\bm\nu_1\\
\end{array}\right)(x/\epsilon,t).
\end{equation*}
According to Section $7$ and Remark \ref{remark 6.1}, it is not difficult to see that the formal approximate solution $(\bm{u}_{\epsilon},\bm{p}_{\epsilon})$ satisfies
\begin{eqnarray}\label{8.2}
 \left\{ \begin{array}{l}\left(\begin{array}{cc}
\bm{u}_{\epsilon}\\
\bm{p}_{\epsilon}\\
\end{array}
\right)_t+\left(\begin{array}{cc}
F & I_n\\
\bar A-F^2 & -F\\
\end{array}
\right)\left(\begin{array}{cc}
\bm{u}_{\epsilon}\\
\bm{p}_{\epsilon}\\
\end{array}
\right)_x=\frac{1}{\epsilon}\left(\begin{array}{cc}
0 & 0\\
0 & -I_n\\
\end{array}
\right)\left(\begin{array}{cc}
\bm{u}_{\epsilon}\\
\bm{p}_{\epsilon}\\
\end{array}
\right)+\epsilon \left(\begin{array}{cc}
0\\
I_n\\
\end{array}
\right)y+\epsilon Y,\vspace{2ex}\\
B\left(\begin{array}{cc}
\bm{u}_{\epsilon}\\
\bm{p}_{\epsilon}\\
\end{array}
\right)(0,t)=b_0(t)+\epsilon b_1(t),\vspace{2ex}\\
\left(\begin{array}{cc}
\bm{u}_{\epsilon}\\
\bm{p}_{\epsilon}\\
\end{array}
\right)(x,0)=\left(\begin{array}{cc}
\bm u_{0}\\
0\\
\end{array}
\right)(x)+\epsilon\left(\begin{array}{cc}
0\\
-(\bar A-F^2)\bm{u}_{0x}\\
\end{array}
\right)(x),\\
 \end{array} \right.
\end{eqnarray}
where
\begin{equation*}
y=y(x,t)=\partial_t\bm{\bar p}_{1}+(\bar A-F^2)\partial_x\bm{\bar u}_{1}-F\partial_x\bm{\bar p}_{1},\quad 
Y=Y(x/\epsilon,t)=\left(
\begin{array}{cc}
\partial_t\bm{\mu}_{1}\\
\partial_t\bm{\nu}_{1}\\
\end{array}
\right).
\end{equation*}

To show the convergence, we make the following assumptions.

\begin{assumption}\label{assumption 8.1}
\end{assumption}

$(1)$ The boundary $x=0$ is non-characteristic  for the conservation laws \eqref{2.1}, that is, the coefficient matrix $F$ is invertible.

$(2)$ The initial data $\bm{u}_0\in H^5( \mathbb{R}^+)$ and boundary data \eqref{2.2} $\hat b\in H^4(0,t_{*}).$

$(3)$ At $(x,t)=(0,0),$ these initial and boundary data are compatible up to order $3.$

\begin{assumption}\label{assumption 8.2}
\end{assumption}

$(1)$ The initial and boundary data in \eqref{8.1} are compatible up to order $2.$

$(2)$ $\bm{p}_{02}\in H^3( \mathbb{R}^+),$ $b_0\in H^4(0,t_{*})$ and $b_1,b_2\in H^3(0,t_{*}).$\vspace{1ex}

Under these assumptions, we use Lemma 7.1 in \cite{zhou2020construction} (see also \cite{metivier2004small} ) and can obtain the following existence result, in which 
$$CH_{t_*}^s=\bigcap_{k\leq s}C^k([0,t_{*}];H^{s-k}(\mathbb R^+))$$
and $H^{k}(\mathbb R^+)$ is the Sobolev space of functions on $\mathbb R^+$  with all derivatives, up to order k, being square-integrable.

\begin{lemma}\label{lemma 8.1}

    $(1)$  The IBVP \eqref{8.1} has an unique solution $(\bm u^\epsilon,\bm p^\epsilon)\in CH_{t_{*}}^3$.\vspace{1ex}

    $(2)$ There is an unique $\bar{\bm u}_0\in CH_{t_{*}}^4$ and an unique $(\bar{\bm u}_1,\bar{\bm p}_1)\in CH_{t_{*}}^2$. Moreover, $\bar{\bm u}_0(0,t)\in H^4(0,t_{*}),$ $\bar{\bm u}_1(0,t)\in H^2(0,t_{*}).$\vspace{1ex}

    $(3)$ ${\bm \mu}_0,{\bm \nu}_0\in H^4([0,t_{*}]\times \mathbb R^+),$ ${\bm \mu}_1,{\bm \nu}_1\in H^2([0,t_{*}]\times \mathbb R^+)$.
\end{lemma} 
\vspace{1ex}

Now we can state the main result of this section.

\begin{theorem}\label{theorem 8.2}
Under the strict sub-characteristic condition, the assumptions \ref{assumption 8.1} and \ref{assumption 8.2}, there exists a constant $C>0$ such that
\begin{equation*}
\max_{t\in[0,t_{*}]}\|\left(\bm u^{\epsilon}-\bm u_{\epsilon},\bm p^{\epsilon}-\bm p_{\epsilon}\right)(\cdot,t)\|_{L^2(\mathbb R^+)}\leq C\epsilon^{\frac{3}{2}}.
\end{equation*}
\end{theorem}

\begin{proof}
Set
\begin{equation*}
W=\left(\begin{array}{cc}
W^I\\
W^{II}\\
\end{array}\right)=\left(\begin{array}{cc}
\bm {u}^{\epsilon}\\
\bm {p}^{\epsilon}\\
\end{array}\right)-\left(\begin{array}{cc}
\bm {u}_{\epsilon}\\
\bm {p}_{\epsilon}\\
\end{array}\right).
\end{equation*}
From \eqref{8.1} and \eqref{8.2}, it follows that $W(x,t)$ satisfies
\begin{eqnarray}\label{8.3}
 \left\{ \begin{array}{l}
 W_t+AW_x=\frac{1}{\epsilon}\left(\begin{array}{cc}
0 & 0\\
0 & -I_n\\
\end{array}
\right)W-\epsilon \left(\begin{array}{cc}
0\\
I_n\\
\end{array}
\right)y-\epsilon Y,\vspace{2ex}\\
BW(0,t)=\epsilon^2b_2(t),\vspace{2ex}\\
W(x,0)=\epsilon^2\left(
\begin{array}{cc}
0\\
\bm{p}_{02}\\
\end{array}
\right)(x),\\
 \end{array} \right.
\end{eqnarray}
where
$$
A=\left(\begin{array}{cc}
F & I_n\\
\bar A-F^2 & -F\\
\end{array}
\right).
$$

 Recall that $F=T\Lambda T^{-1}$ and $\bar A=T\bar\Lambda T^{-1}$. Set
\begin{equation*} 
A_0= \left(
\begin{array}{cc}
T^{-*}(\bar\Lambda-\Lambda^2)T^{-1} &\quad     0\\
0      &\quad  T^{-*}T^{-1}   \\
\end{array}
\right).
\end{equation*} 
It is not difficult to verify that $A_0A$ is symmetric and $A_0$ is symmetric positive definite under the strict sub-characteristic condition $a_j>\lambda_j^2.$ This indicates that the system in \eqref{8.3} satisfies the structural stability condition~\cite{yong1999singular} and is symmetrizable hyperbolic. 
Therefore, there exists an invertible matrix $L$ such that
\begin{equation*} 
 LAL^{-1}=\text{diag}(\sqrt a_1,\sqrt a_2,\cdots,\sqrt a_n,-\sqrt a_1,-\sqrt a_2,\cdots, -\sqrt a_n)\equiv D.
\end{equation*}

Denote by $L_{\pm}$ the first~(last) $n$ rows of $L$. We follow \cite{gustafsson1995time} and decompose the solution $W(x,t)$ as
$$
W(x,t)=W_1(x,t)+W_2(x,t).
$$
Here $W_1=W_1(x,t)$ solves 
\begin{eqnarray}\label{8.4}
 \left\{ \begin{array}{l}
 W_{1t}+AW_{1x}=\frac{1}{\epsilon}\left(\begin{array}{cc}
0 & 0\\
0 & -I_n\\
\end{array}
\right)W_1-\epsilon \left
 (\begin{array}{cc}
0\\
I_n\\
\end{array}
\right)y-\epsilon Y,\vspace{1ex}\\
L_+W_1(0,t)=0,\vspace{1ex}\\
W_1(x,0)=\epsilon^2\left(
\begin{array}{cc}
0\\
\bm{p}_{02}\\
\end{array}
\right)(x),
 \end{array} \right.
\end{eqnarray}
while $W_2=W_2(x,t)$ satisfies
\begin{eqnarray}\label{8.5}
 \left\{ \begin{array}{l}
 W_{2t}+AW_{2x}=\frac{1}{\epsilon}\left(
\begin{array}{cc}
0 & 0\\
0 & -I_n\\
\end{array}
\right)W_2,\vspace{2ex}\\
BW_2(0,t)=-BW_1(0,t)+\epsilon^2b_2(t),\vspace{2ex}\\
W_2(x,0)=0.\\
 \end{array} \right.
\end{eqnarray}
It is known from \cite{gustafsson1995time} that the BC in \eqref{8.4} satisfies the Uniform Kreiss Condition. According to the existence theory in \cite{benzoni2007multi}, there exists a unique solution $W_1\in C([0,t_{*}];L^2(\mathbb R^{+})).$ In addition, by Remark $3.2$ in \cite{yong1999boundary}, the Uniform Kreiss Condition is implied by the GKC. Thus, the existence theory in \cite{benzoni2007multi} indicates that there exists a unique solution $W_2\in C([0,\infty);L^2(\mathbb R^{+}))$ provided that $W_1(0,t)$ and $b_2(t)$ in \eqref{8.5} are replaced by their zero-extensions. 
\vspace{1ex}   

For $W_1$, we multiply \eqref{8.4} with $W_1^*A_0$ from right to get
\begin{align*}
(W_1^*A_0W_1)_t+(W_1^*A_0AW_1)_x & =\frac{2}{\epsilon}W_1^*A_0\left(\begin{array}{cc}
0 & 0\\
0 & -I_n\\
\end{array}
\right)W_1-2\epsilon W_1^*A_0\left(
\begin{array}{cc}
0\\
I_n\\
\end{array}
\right)y-2\epsilon W_1^*A_0Y \nonumber\\[3mm]
& \leq-c_0\frac{|W_1^{II}|^2}{\epsilon}+C\epsilon^3|y|^2+C|W_1|^2+C\epsilon^2|Y|^2,
\end{align*}
where $W_1^{II}$ stands for the last $n$ components of $W_1$, $C>0$ is a generic constant and $c_0>0$ is a small constant.
Integrating the last inequality over $x\in[0,\infty)$ we have
\begin{align}\label{8.6}
&\frac{d}{dt}\int_0^{\infty}W_1^*A_0W_1dx+c_0\int_0^{\infty}\frac{|W_1^{II}|^2}{\epsilon}dx-W_1^*A_0AW_1|_{x=0}\nonumber\\[4mm]
\leq& C\left(\int_0^{\infty}\epsilon^3|y|^2dx+\int_0^{\infty}|W_1|^2dx+\int_0^{\infty}\epsilon^2|Y(x/\epsilon,t)|^2dx\right)\nonumber\\[4mm]
=& C\left(\int_0^{\infty}\epsilon^3|y|^2dx+\int_0^{\infty}|W_1|^2dx+\epsilon^3\int_0^{\infty}|Y(\xi,t)|^2d\xi\right).
\end{align}
For the boundary term, we notice that $L^{-*}A_0L^{-1}$ is block-diagonal~\cite{herty2016feedback} and have
\begin{align*}
-W_1^*A_0AW_1\Big|_{x=0} 
                    & = -W_1^*(0,t)A_0L^{-1}D LW_1(0,t)\nonumber\\[2mm]
                    & = -W_1^*(0,t)L^{*}L^{-*}A_0L^{-1}DLW_1(0,t)\nonumber\\[2mm]
                     & \geq c_1|L_{-}W_1(0,t)|^2
\end{align*}
with $c_1$ a positive constant. In addition, we have 
$C^{-1}W_1^*A_0W_1\leq|W_1|^2\leq CW_1^*A_0W_1$ 
due to the positiveness of $A_0$. Consequently, applying  Gronwall's inequality to \eqref{8.6} we obtain
\begin{align}\label{8.7}
\|W_1(\cdot,t)\|^2_{L^2(\mathbb R^+)}
\leq  Ce^{Ct_{*}}\left(\|W_1(\cdot,0)\|^2_{L^2(\mathbb R^+)}+\epsilon^3\int_0^{t_{*}}\int_0^{\infty}|y|^2dx+\epsilon^3
       \int_0^{t_{*}}\int_0^{\infty}|Y(\xi,t)|^2d\xi\right).
\end{align}

From  Assumption \ref{assumption 8.2} and Lemma \ref{lemma 8.1}, we know that 
\begin{align*} 
&\bm{p}_{02}\in L^2(\mathbb R^+),\qquad\quad \bm{\mu}_1,\bm{\nu}_1\in H^1([0,t_{*}]\times\mathbb R^+),\\[2mm]
&y=\partial_t\bm{\bar p}_{1}+(\bar A-F^2)\partial_x\bm{\bar u}_{1}-F\partial_x\bm{\bar p}_{1}\in L^2([0,t_{*}]\times\mathbb R^+).
\end{align*}
Thus, it follows that 
\begin{align*} 
&\|W_1(\cdot,0)\|^2_{L^2(\mathbb R^+)}=\int_0^{\infty}|\epsilon^2\bm{p}_{02}|^2dx\leq C\epsilon^4,\\[2mm]
&\int_0^{t_{*}}\int_0^{\infty}|y(x,t)|^2dx\leq C,\qquad
\int_0^{t_{*}}\int_0^{\infty}|Y(\xi,t)|^2d\xi\leq C.
\end{align*}
Combining these with \eqref{8.7}, we get
\begin{equation*} 
\|W_1(\cdot,t)\|^2_{L^2(\mathbb R^+)}\leq Ce^{Ct_{*}}\epsilon^{3}.
\end{equation*}
Then we integrate \eqref{8.6} over $t\in[0,t_{*}]$  to obtain
\begin{equation}\label{8.8}
\|W_1\|^2_{C([0,t_{*}];L^2(\mathbb R^+))}+\frac{1}{\epsilon}\|W_1^{II}\|^2_{L^2([0,t_{*}]\times \mathbb R^+)}+\|L_{-}W_1|_{x=0}\|^2_{L^2([0,t_{*}])}\leq Ce^{Ct_{*}}\epsilon^{3}.\\[2mm]
\end{equation}

 Next we follow \cite{gustafsson1995time} to estimate $W_2$. 
 Denote by $\hat W_2=\hat W_2(x,\xi_0)$ the Laplace transform of $W_2=W_2(x,t)$ with respect to time $t$. It follows from \eqref{8.5} that 
\begin{eqnarray}\label{8.9}
 \left\{ \begin{array}{l}
\hat W_{2x}=A^{-1}(\eta S-\xi_0I_n)\hat W_2\equiv M(\eta,\xi_0)\hat W_2,\vspace{2ex}\\
B\hat W_2(0,\xi_0)=\epsilon^2\hat b_2(\xi_0)-B\hat W_1(0,\xi_0),\vspace{2ex}\\
\|\hat W_2(\cdot,\xi_0)\|_{L^2(\mathbb R^+)}<\infty,\qquad \text{for a.e.}\quad \xi_0,
 \end{array} 
 \right.
\end{eqnarray}
where $\eta=1/\epsilon$ and $S=\text{diag}(0,-I_n)$.\vspace{1ex} 

Recall from Lemma $2.3$ in \cite{yong1999boundary} that, under the sub-characteristic condition, $M=M(\eta,\xi_0)$ has $n$ stable eigenvalues and $n$ unstable eigenvalues for all $\eta\geq0$ and all $\xi_0$ with $\text{Re}\xi_0>0$. By $Shur^{,}s$ theorem, there exists an unitary matrix $U$ such that
\begin{equation*}
U^{*}MU=\left(
\begin{array}{cc}
M_{11} & M_{12} \\
  0    & M_{22} \\
\end{array}
\right),
\end{equation*}
where the $M_{11}$ is a stable $n\times n$-matrix and the $M_{22}$ is a unstable $n\times n$-matrix. Set $\varphi=U^{*}\hat W_{2}=\left(
\begin{array}{c}
\varphi_1 \\
\varphi_2 \\
\end{array}
\right).$
The equation in \eqref{8.9} becomes
\begin{equation*}
\partial_x \left(
\begin{array}{c}
\varphi_1(x,\xi_0)\\
\varphi_2(x,\xi_0)\\
\end{array}
\right)=\left(
\begin{array}{cc}
M_{11} & M_{12}\\
   0   & M_{22}\\
\end{array}
\right)\left(
\begin{array}{cc}
\varphi_1(x,\xi_0)\\
\varphi_2(x,\xi_0)\\
\end{array}
\right).
\end{equation*}
The bounded solution to the last equation is $\varphi_2=0$ and 
\begin{eqnarray*}
\varphi_1(x,\xi_0) & = & e^{M_{11}x}\varphi_1(0,\xi_0).
\end{eqnarray*}
The corresponding BC becomes
\begin{equation*}
B\hat W_2(0,\xi_0)=BU_{I}\varphi_1(0,\xi_0)+BU_{II}\varphi_2(0,\xi_0)=\epsilon^2\hat b_2(\xi_0)-B\hat W_1(0,\xi_0),
\end{equation*}
where $U=(U_{I},U_{II}).$ Thus we have 
\begin{eqnarray*}
BU_{I}\varphi_1(0,\xi_0)& = & \epsilon^2\hat b_2(\xi_0)-B\hat W_1(0,\xi_0).
\end{eqnarray*}

Notice that $(BU_{I})^{-1}$ is uniformly bounded due to the GKC. We have
\begin{align*}
|\varphi_1(0,\xi_0)| & =\left|(BU_{I})^{-1}[\epsilon^2\hat b_2(\xi_0)-B\hat W_1(0,\xi_0)]\right|\\[2mm]
                     & \leq C\left(|\epsilon^2\hat b_2(\xi_0)|+|\hat W_1(0,\xi_0)|\right).
\end{align*}
Since $U$ is a unitary matrix,  it is easy to see that 
\begin{align*}
|\hat W_2(0,\xi_0)|\leq C\left(\epsilon^2|\hat b_2(\xi_0)|+|\hat W_1(0,\xi_0)|\right).
\end{align*}
According to the Parseval equality, the last inequality leads to
\begin{align*}
\int_0^{\infty}e^{-2t\text{Re}\xi_0}|W_2(0,t)|^2dt
     &  \leq C\left(\int_0^{\infty}e^{-2t\text{Re}\xi_0}|\epsilon^2 b_2(t)|^2dt+\int_0^{\infty}e^{-2t\text{Re}\xi_0}|W_1(0,t)|^2dt\right)\\[3mm]
     &   \leq C\left(\epsilon^4\int_0^{\infty}| b_2(t)|^2dt+\int_0^{\infty}|W_1(0,t)|^2dt\right).
\end{align*}
Because  the right-hand side is independent of $\text{Re}\xi_0,$ then  we have
\begin{align*}
\int_0^{\infty}|W_2(0,t)|^2dt\leq C\left(\epsilon^4\int_0^{\infty}| b_2(t)|^2dt+\int_0^{\infty}|W_1(0,t)|^2dt\right).
\end{align*}
By a standard argument in \cite{gustafsson1995time}, the last inequality implies
\begin{align}\label{8.10}
\int_0^{t_{*}}|W_2(0,t)|^2dt
&\leq C\left(\epsilon^4\| b_2(t)\|_{L^2([0,t_{*}])}^2+\|W_1(0,t)\|_{L^2([0,t_{*}])}^2\right)\nonumber\\[2mm]
&\leq C\left(\epsilon^4+\epsilon^3\right).
\end{align}
Here the estimate \eqref{8.8} has been used.
\vspace{1ex}

Finally, we multiply \eqref{8.5} with $W_2^*A_0$ from right to obtain
\begin{align*}
(W_2^*A_0W_2)_t+(W_2^*A_0AW_2)_x & =\frac{2}{\epsilon}W_2^*A_0\left(
\begin{array}{cc}
0 & 0\\
0 & -I_n\\
\end{array}
\right)W_2\leq 0.
\end{align*}
Integrating the above inequality over $(x,t)\in[0,\infty)\times[0,t_{*}]$ and using \eqref{8.10}, we get
\begin{align*}
\max_{t\in[0,t_{*}]}\|W_2(\cdot,t)\|_{L^2(\mathbb R^+)}^2\leq C\int_0^{t_{*}}|W_2(0,t)|^2dt \leq C\epsilon^3.
\end{align*}
This together with \eqref{8.8} completes the proof.\vspace{1ex}
\end{proof}

Furthermore, we have the following $H^{1}$-estimate.
\begin{theorem}
Under the strict sub-characteristic condition, assumptions \ref{assumption 8.1} and \ref{assumption 8.2}, there exists a constant $C>0$ such that
\begin{equation*}
\max_{t\in[0,t_{*}]}\|\left(\bm u^{\epsilon}-\bm u_{\epsilon},\bm p^{\epsilon}-\bm p_{\epsilon}\right)(\cdot,t)\|_{H^1(\mathbb R^+)}\leq C\epsilon^{\frac{1}{2}}.
\end{equation*}
\end{theorem}
\begin{proof}
We firstly estimate $W_t.$ From \eqref{8.3}, it is easy to see that $W_t=W_t(x,t)$ satisfies
\begin{eqnarray*}
 \left\{ \begin{array}{l}
 (W_t)_t+A(W_t)_x=\frac{1}{\epsilon}\left(\begin{array}{cc}
0 & 0\\
0 & -I_n\\
\end{array}
\right)W_t-\epsilon \left(\begin{array}{cc}
0\\
I_n\\
\end{array}
\right)\partial_ty-\epsilon \partial_tY,\vspace{2ex}\\\
BW_t(0,t)=\epsilon^2\partial_tb_2(t),\vspace{2ex}\\
W_t(x,0)=-AW_x(x,0)-\epsilon\left(\begin{array}{cc}
0\\
\bm{p}_{02}\\
\end{array}
\right)(x,0)-\epsilon\left(\begin{array}{cc}
0\\
I_n\\
\end{array}
\right) y(x,0)-\epsilon Y(x/\epsilon,0).\\
 \end{array} \right.
\end{eqnarray*}
For the initial data of $W_t,$ we use Assumption \ref{assumption 8.2} and Lemma \ref{lemma 8.1} to get
\begin{align*}
\|W_t(\cdot,0)\|^2_{L^2(\mathbb R^+)}
   & \leq  C\left(\|W_x(\cdot,0)\|^2_{L^2(\mathbb R^+)}+\epsilon^2\|\bm{p}_{02}\|^2_{L^2(\mathbb R^+)}+\epsilon^2\|y(\cdot,0)\|^2_{L^2(\mathbb R^{+})}+\epsilon^2\|Y(\cdot/\epsilon,0)\|^2_{L^2(\mathbb R^{+})}\right)\\[2mm]
   & = C\left(\epsilon^4\|\partial_x\bm{p}_{02}\|^2_{L^2(\mathbb R^+)}+\epsilon^2\|\bm{p}_{02}\|^2_{L^2(\mathbb R^+)}+\epsilon^2\|y(\cdot,0)\|^2_{L^2(\mathbb R^{+})}+\epsilon^3\|Y(\cdot,0)\|_{L^2(\mathbb R^{+})}\right)\\[2mm]
   &\leq C\epsilon^2.
\end{align*}
In analogue to the estimate for $W$ in Theorem \ref{theorem 8.2}, we obtain
\begin{align*}
\|W_t(\cdot,t)\|^2_{L^2(\mathbb R^+)}
 \leq &C\epsilon^4\|\partial_tb_2(t)\|_{L^2([0,t_{*}])}^2+ C\epsilon^2\nonumber\\[3mm]
 &+Ce^{Ct{*}}\left(\epsilon^3\int_0^{t{*}}\int_0^{\infty}|y_t(x,t)|^2dx+\epsilon^3\int_0^{t{*}}\int_0^{\infty}|Y_t(\xi,t)|^2d\xi\right).
\end{align*} 
Moreover, since $b_2\in H^1(0,t_{*})$, it follows from the last inequality that
$$
\|W_t(\cdot,t)\|^2_{L^2(\mathbb R^+)}\leq C\epsilon^2.
$$

Next we estimate $W_x$ in terms of equation in~\eqref{8.3}. Because the matrix $A$ is invertible, then
\begin{align*}
&\|W_x(\cdot,t)\|_{L^2(\mathbb R^+)}\\[2mm]
 =&\Big\|A^{-1}\left(-W_t(\cdot,t)+\frac{1}{\epsilon}\left(\begin{array}{cc}
0 & 0\\
0 & -I_n\\
\end{array}
\right)W(\cdot,t)-\epsilon \left(\begin{array}{cc}
0\\
I_n\
\end{array}
\right)y(\cdot,t)-\epsilon Y(\cdot/\epsilon,t)\right)\Big\|_{L^2(\mathbb R^+)} \\[1mm]
 \leq& C\left(\|W_t(\cdot,t)\|_{L^2(\mathbb R^+)}+\frac{1}{\epsilon}\|W(\cdot,t)\|_{L^2(\mathbb R^+)}+\epsilon\|y(\cdot,t)\|_{L^2(\mathbb R^+)}+\epsilon\|Y(\cdot/\epsilon,t)\|_{L^2(\mathbb R^+)}\right)\\[1mm]
\leq& C\epsilon^{\frac{1}{2}}.
\end{align*}
These together with the estimate of $W$ lead to
\begin{equation*}
\max_{t\in[0,t_{*}]}\|\left(\bm u^{\epsilon}-\bm u_{\epsilon},\bm p^{\epsilon}-\bm p_{\epsilon}\right)(\cdot,t)\|_{H^1(\mathbb R^+)}\leq C\epsilon^{\frac{1}{2}}.
\end{equation*}
This completes the proof.
\end{proof}

At the end of this section, we deduce from Theorem \ref{theorem 8.2} that
\begin{align*}
\|\bm u^{\epsilon}(\cdot,t)-\bm {\bar u}_0(\cdot,t)\|_{L^2(\mathbb R^+)}
   & \leq \|\bm u_{\epsilon}(\cdot,t)-\bm {\bar u}_0(\cdot,t)\|_{L^2(\mathbb R^+)}+C\epsilon^{\frac{3}{2}}\\[2mm]
   & = \|\epsilon\bm u_1(\cdot,t)+\bm{\mu}_0(\cdot/\epsilon,t))+\epsilon\bm{\mu}_1(\cdot/\epsilon,t))\|_{L^2(\mathbb R^+)}+C\epsilon^{\frac{3}{2}}\\[2mm]
   & \leq C\epsilon^{\frac{1}{2}},
\end{align*}
where $\bm {\bar u}_0(x,t)$ is the solution to the IBVP of the conservation laws \eqref{2.1} and $\bm {u}^{\epsilon}(x,t)$ is the solution to the IBVP \eqref{8.1}.

\section*{Appendix}
This appendix presents the contents of Sections 5-7 for the simple case where $l=n$. In this case, the constructed BC in \eqref{3.4} reads as
\begin{equation*}
   B_u=\hat B=I_n,\quad B_p \quad  \text{arbitrary} \quad \text{and} \quad b_0(t)=\hat b(t).
\end{equation*}
The GKC can be easily verified under certain constraints on $ B_p.$

\begin{prop}\label{prop 8.1}
Under the strict sub-characteristic condition, the GKC holds if the spectral radius  
$\rho(B_p)<\frac{1}{\max\limits_{j}(\sqrt 2+1)\sqrt a_j}$  or  $T^{-1}B_pT$
 is  a lower~(upper) triangle matrix with its $j$-th diagonal element 
$\delta_{j}>\frac{1}{\lambda_j-\sqrt a_j}.$
\end{prop}

\begin{proof}

From \eqref{RMS} we have
$BR_{M}^{S}(\eta,\xi_0) = T+ B_pTQ.$ Thus the determinant $BR_{M}^{S}(\eta,\xi_0)$ is far from zero for the spectral radius  
$\rho(Q)=\max\limits_{j}(\sqrt 2+1)\sqrt a_j$ due to Lemma \ref{lemma 4.1} and the condition $\rho(B_p)<\frac{1}{\max\limits_{j}(\sqrt 2+1)\sqrt a_j}$.

Under the other condition, we may as well assume that $T^{-1}B_pT$ is upper triangle. Thus we deduce that
\begin{align*}
 B R_{M}^{S}(\eta,\xi_0) & = T+ B_pTQ\\[2mm]
                                     & = T\left [I_n+
    \left( \begin{matrix}
        \begin{matrix}
          \delta_{1}q_{1} &        \\
                 & \delta_{2}q_{2}
        \end{matrix}
         & \textup{\Huge{*}} 
        \\
        \textup{\Huge{0}}
         &
        \begin{matrix}
          \ddots &        \\
                 & \delta_{n}q_{n}
        \end{matrix}
      \end{matrix} \right)\right]
\end{align*}
and 
\begin{align*}
\text{det}\{ B R_{M}^{S}(\eta,\xi_0)\}=|T|\prod_{j=1}^n(1+\delta_{j}q_{j}).
\end{align*}
For $\delta_{j}= 0$ or $\delta_{j}\neq 0$ but 
\begin{equation}\label{8.13}
 -\frac{1}{\delta_j} \notin \text{closure}\{q_j(\eta,\xi_0):\text{Re}\xi_0>0,\eta\geq 0\},
\end{equation}
it is clear that $\left|1+\delta_{j}q_{j}\right|$ has a positive lower bound. 
On the other hand, we refer to the proof of Lemma \ref{lemma 4.1} and know that the intersection of the closure  and the real axis is $[0,\sqrt a_j-\lambda_j].$ Thus the condition \eqref{8.13} is satisfied if $\delta_j$ is a real number and 
$$
-\frac{1}{\delta_{j}}\notin [0,\sqrt a_j-\lambda_j].
$$
The latter holds if $\delta_{j}>\frac{1}{\lambda_j-\sqrt a_j}.$ This completes the proof.
\end{proof}

As to the compatibility, at $(x,t)=(0,0)$, of the initial and boundary data for the relaxation system \eqref{2.3}, we follow the discussion in Section $6$. Thus, we only need to check the relation \eqref{6.6}. When $B_u=\hat B=I_n$ and $b_0(t)=\hat b(t)$, \eqref{6.6} is just the assumption \eqref{6.1} .\vspace{1ex}

To determine the coefficients $\bm{\bar u}_0,$ $\bm{\bar u}_1,$ $\bm{\nu}_0$  and $\bm{\nu}_1$ of the formal asymptotic solution, we observe that the coefficient matrix $F\bar A^{-1}$ in Equation \eqref{2.17} has only positive eigenvalues. Then the unique bounded solution thereof is $\bm{\nu}_0=0$. Similarly,  we have $\bm{\nu}_1=0$ due to \eqref{2.19}. 
On the other hand, the BCs for $\bm{\bar u}_0$ and $\bm{\bar u}_1$ at boundary $x=0$ can be obtained as follows. From the expectation that the asymptotic solution  satisfies the BC~\eqref{3.1}:
\begin{equation*}
(B_u,B_p)\left(
\begin{array}{cc}
\bm{\bar u}_0(0,t)\vspace{1ex}\\
\bm{\bar p}_0(0,t)\\
\end{array}
\right)=b_0(t),\qquad
(B_u,B_p)\left(
\begin{array}{cc}
\bm{\bar u}_1(0,t)\vspace{1ex}\\
\bm{\bar p}_1(0,t)\\
\end{array}
\right)=b_1(t),
\end{equation*}
it follows from \eqref{2.8} and \eqref{2.10} that
\begin{equation*}
B_u\bm{\bar u}_0(0,t)=b_0(t),\qquad
B_u\bm{\bar u}_1(0,t)=b_1(t)+B_p(\bar A-F^2)\bm{\bar u}_{0x}(0,t).
\end{equation*}
Since $B_u=I_n$ and $b_0(t)=\hat b(t)$, we obtain
\begin{equation}\label{8.14}
\bm{\bar u}_0(0,t)=\hat b(t),\qquad \bm{\bar u}_1(0,t)=b_1(t)+B_p(\bar A-F^2)\bm{\bar u}_{0x}(0,t).
\end{equation}
Moreover, the choice \eqref{7.1} of initial data for $\bm{\bar u}_0$ and $\bm{\bar u}_1$ is still valid in this case.
In this way, $\bar{\bm u}_0$ can be uniquely obtained by solving the IBVP of the conservation laws \eqref{2.1}, while $\bar{\bm u}_1$ can be solved from the IBVP \eqref{2.11}, \eqref{7.1} and \eqref{8.14}.

\bibliographystyle{plain}
\bibliography{ref}

\end{document}